\documentclass{article}
\usepackage{amsmath,amssymb,amsthm,url}
\usepackage{mathrsfs}
\usepackage[colorlinks=true]{hyperref}
\usepackage{pdfsync}
\usepackage{stmaryrd}

\usepackage[applemac]{inputenc}
\usepackage[T1]{fontenc}

\def\R{\textrm{I\kern-0.21emR}}
\def\N{\textrm{I\kern-0.21emN}}

\newcommand{\C} {\mathbb{C}}

\renewcommand{\geq}{\geqslant}
\renewcommand{\leq}{\leqslant}

\newtheorem{theorem}{Theorem}  

\newtheorem{corollary}{Corollary}

\newtheorem{lemma}{Lemma}
\newtheorem{remark}{Remark}
\theoremstyle{definition}\newtheorem{example}{Example}

\title{Nonlinear damped partial differential equations and their uniform discretizations}

\author{Fatiha Alabau-Boussouira\footnote{Universit\'e de Lorraine, IECL, UMR 7502, 57045 Metz Cedex 1, France (\texttt{fatiha.alabau@univ-lorraine.fr}) En d\'el\'egation CNRS au Laboratoire Jacques-Louis Lions, UMR 7598.}
\and
Yannick Privat\footnote{CNRS, Sorbonne Universit\'es, UPMC Univ Paris 06, UMR 7598, Laboratoire Jacques-Louis Lions, F-75005, Paris, France (\texttt{yannick.privat@upmc.fr}).}
\and
Emmanuel Tr\'elat\footnote{Sorbonne Universit\'es, UPMC Univ Paris 06, CNRS UMR 7598, Laboratoire Jacques-Louis Lions, and Institut Universitaire de France, F-75005, Paris, France (\texttt{emmanuel.trelat@upmc.fr}).} 
}

\date{}

\begin{document}

\maketitle

\begin{abstract}
We establish sharp energy decay rates for a large class of nonlinearly first-order damped systems, and we design discretization schemes that inherit of the same energy decay rates, uniformly with respect to the space and/or time discretization parameters, by adding appropriate numerical viscosity terms. Our main arguments use the optimal-weight convexity method and uniform observability inequalities with respect to the discretization parameters. We establish our results, first in the continuous setting, then for space semi-discrete models, and then for time semi-discrete models. The full discretization is inferred from the previous results.

Our results cover, for instance, the Schr\"odinger equation with nonlinear damping, the nonlinear wave equation, the nonlinear plate equation, the nonlinear transport equation, as well as certain classes of equations with nonlocal terms.
\end{abstract}


\medskip


\tableofcontents

\section{Introduction}\label{sec1}
Let $X$ be a Hilbert space. Throughout the paper, we denote by $\Vert\cdot\Vert_X$ the norm on $X$ and by $\langle\cdot,\cdot\rangle_X$ the corresponding scalar product. Let $A:D(A)\rightarrow X$ be a densely defined skew-adjoint operator, and let $B:X\rightarrow X$ be a nontrivial bounded selfadjoint nonnegative operator.
Let $F:X\rightarrow X$ be a (nonlinear) mapping, assumed to be Lipschitz continuous on bounded subsets of $X$.
We consider the differential system
\begin{equation}\label{main_eq}
u'(t) + A u(t) + B F(u(t))  = 0 .
\end{equation}
If $F=0$ then the system \eqref{main_eq} is conservative, and for every $u_0\in D(A)$, there exists a unique solution $u(\cdot)\in C^0(0,+\infty;D(A))\cap C^1(0,+\infty;X)$ such that $u(0)=u_0$, which satisfies moreover $\Vert u(t)\Vert_X=\Vert u(0)\Vert_X$, for every $t\geq 0$. 

If $F\neq 0$ then the system \eqref{main_eq} is expected to be dissipative if the nonlinearity $F$ has ``the good sign". Defining the energy of a solution $u$ of \eqref{main_eq} by
\begin{equation}\label{defE}
E_u(t) = \frac{1}{2}\Vert u(t)\Vert_X^2 ,
\end{equation}
we have, as long as the solution is well defined,
\begin{equation}\label{Eprime}
E_u'(t) = -\langle u(t),BF(u(t))\rangle_X . 
\end{equation}
In the sequel, we will make appropriate assumptions on $B$ and on $F$ ensuring that $E_u'(t)\leq 0$. It is then expected that the solutions are globally well defined and that their energy  decays asymptotically to $0$ as $t\rightarrow +\infty$.

\medskip

The objective of this paper is twofold.

First of all, in Section \ref{sec:Continuous} we provide adequate assumptions under which the solutions of \eqref{main_eq} have their norm decaying asymptotically to $0$ in a quasi-optimal way. This first result, settled in an abstract continuous setting, extends former results of \cite{alabau_ammari} (established for damped wave equations) to more general equations and damping operators for stabilization issues based on indirect arguments (for direct arguments see e.g. \cite{alabau_AMO2005, alabauJDE}).

Then, in Section \ref{sec_discretization}, we investigate discretization issues, with the objective of proving that, for appropriate discretization schemes, the discrete approximate solutions have a uniform decay.
In Section \ref{sec_semidiscrete}, we first consider spatial semi-discrete approximation schemes, and in Section \ref{sec_semidiscretetime} we deal with time semi-discretizations. The full discretization is done in Section \ref{sec_fullydiscrete}.
In all cases, we establish uniform asymptotic decay with respect to the mesh size, by adding adequate viscosity terms in the approximation schemes.

\section{Continuous setting}\label{sec:Continuous}
\subsection{Main result}\label{sec:mainResConti}

\paragraph{Assumptions and notations.}
First of all, we assume that
\begin{equation}\label{ineqBF}
\langle u, B F(u)\rangle_X \geq 0,
\end{equation}
for every $u\in X$. Using \eqref{Eprime}, this first assumption ensures that the energy $E_u(t)$ defined by \eqref{defE} is nonincreasing.

Since $B$ is bounded, nonnegative and selfadjoint on $X$, it follows from the well-known spectral theorem that $B$ is unitarily equivalent to a multiplication (see, e.g., \cite{Kato}). More precisely, there exist a probability space $\Omega$ with measure $\mu$, a real-valued bounded nonnegative measurable function $b$ defined on $\Omega$ (satisfying $\Vert b\Vert_{L^\infty(\Omega,\mu)} = \Vert B\Vert$), and an isometry $U$ from $L^2(\Omega,\mu)$ into $X$, such that
$$
(U^{-1} B U f)(x) = b(x) f(x),
$$
for almost every $x\in\Omega$ and for every $f\in L^2(\Omega,\mu)$.
Another usual way of writing $B$ is
$$
B = \int_{0}^{+\infty} \lambda \, dE(\lambda),
$$
where the family of $E(\lambda)$ is the family of spectral projections associated with $B$. We recall that the spectral projections are obtained as follows.
Defining the orthogonal projection operator $Q_\lambda$ on $L^2(\Omega,\mu)$ by $(Q_\lambda f)(x) = \chi_{\{b(x)\leq\lambda\}} (x) f(x)$, for every $f\in L^2(\Omega,\mu)$ and for every $x\in\Omega$, we have $E(\lambda)=U Q_\lambda U^{-1}$.

We define the (nonlinear) mapping $\rho: L^2(\Omega,\mu)\rightarrow L^2(\Omega,\mu)$ by
$$
\rho(f) = U^{-1} F( Uf ),
$$
for every $f\in L^2(\Omega,\mu)$. In other words, the mapping $\rho$ is equal to the mapping $F$ viewed through the isometry $U$.

Note that, setting $f=U^{-1}u$, the equation \eqref{main_eq} is equivalent to
$\partial_t f + \bar A f + b \rho(f) = 0$, with $\bar A=U^{-1}AU$, densely defined skew-adjoint operator on $L^2(\Omega,\mu)$, of domain $U^{-1}D(A)$.

\medskip

We assume that $\rho(0)=0$ and that 
\begin{equation}\label{assumptionfrhof}
f \rho(f) \geq 0, 
\end{equation}
for every $f\in L^2(\Omega,\mu)$.
Following \cite{alabau_AMO2005, alabauJDE, alabau_ammari}, we assume that there exist $c_1>0$ and $c_2>0$ such that, for every $f\in L^\infty(\Omega,\mu)$,
\begin{equation}\label{ineqg}
\begin{split}
c_1\, g(\vert f(x)\vert) &\leq \vert\rho(f)(x)\vert \leq c_2\, g^{-1}(\vert f(x)\vert) \quad \textrm{for almost every}\ x\in\Omega\ \textrm{such that}\ \vert f(x)\vert\leq 1,    \\
c_1\, \vert f(x)\vert & \leq \vert\rho(f)(x)\vert \leq c_2 \, \vert f(x)\vert\qquad\quad\ \textrm{for almost every}\ x\in\Omega\ \textrm{such that}\ \vert f(x)\vert\geq 1 ,
\end{split}
\end{equation}
where $g$ is an increasing odd function of class $C^1$ such that 
$$
g(0)=g'(0)=0, \qquad \frac{sg'(s)^2}{g(s)}\xrightarrow[s\to 0]{} 0
$$ 
and such that the function $H$ defined by $H(s)=\sqrt{s}g(\sqrt{s})$, for every $s\in [0,1]$, is strictly convex on $[0,s_0^2]$ for some $s_0\in(0,1]$ (chosen such that $g(s_0)<1$). 

We define the function $\widehat H$ on $\R$ by $\widehat H(s)=H(s)$ for every $s\in[0,s_0^2]$ and by $\widehat H(s)=+\infty$ otherwise. Using the convex conjugate function $\widehat H^*$ of $\widehat H$, we define the function $L$ on $[0,+\infty)$ by $L(0)=0$ and, for $r>0$, by
\begin{equation}\label{defL}
L(r)= \frac{\widehat H^*(r)}{r} = \frac{1}{r} \sup_{s\in \R}\left(rs-\widehat H(s)\right).
\end{equation}
By construction, the function $L:[0,+\infty)\rightarrow[0,s_0^2)$ is continuous and increasing.
We define the function $\Lambda_H:(0,s_0^2]\rightarrow(0,+\infty)$ by $\Lambda_H(s)=H(s)/sH'(s)$, and for $s\geq 1/H'(s_0^2)$ we set
\begin{equation}\label{defpsir}
\psi(s)=\frac{1}{H'(s_0^2)}+\int_{1/s}^{H'(s_0^2)}\frac{1}{v^2(1-\Lambda_H((H')^{-1}(v)))}\,dv.
\end{equation}
The function $\psi:[1/H'(s_0^2),+\infty)\rightarrow[0,+\infty)$ is continuous and increasing.


\medskip

Throughout the paper, we use the notations $\lesssim$ and $\simeq$ in many estimates, with the following meaning.
Let $\mathcal{S}$ be a set, and let $F$ and $G$ be nonnegative functions defined on $\R\times\Omega\times\mathcal{S}$.
The notation $F\lesssim G$ (equivalently, $G\gtrsim F$) means that there exists a constant $C>0$, only depending on the function $g$ or on the mapping $\rho$, such that $F(t,x,\lambda)\leq C G(t,x,\lambda)$ for all $(t,x,\lambda)\in\R\times\Omega\times\mathcal{S}$. 
The notation $F_1\simeq F_2$ means that $F_1\lesssim F_2$ and $F_1\gtrsim F_2$.

In the sequel, we choose $\mathcal{S} = X$, or equivalently, using the isometry $U$, we choose $\mathcal{S} = L^2(\Omega,\mu)$, so that the notation $\lesssim$ designates an estimate in which the constant does not depend on $u\in X$, or on $f\in L^2(\Omega,\mu)$, but depends only on the mapping $\rho$. We will use these notations to provide estimates on the solutions $u(\cdot)$ of \eqref{main_eq}, meaning that the constants in the estimates do not depend on the solutions.

For instance, the inequalities \eqref{ineqg} can be written as
\begin{equation*}
\begin{split}
g(\vert f\vert) \lesssim \vert\rho(f)\vert \lesssim g^{-1}(\vert f\vert) \quad &\textrm{on the set}\ \vert f\vert\lesssim 1,    \\
\vert\rho(f)\vert \simeq \vert f\vert\qquad\quad\ &\textrm{on the set}\ \vert f\vert\gtrsim 1 .
\end{split}
\end{equation*}

The main result of this section is the following.

\begin{theorem}\label{thm_continuous}
In addition to the above assumptions, we assume that there exist $T>0$ and $C_T>0$ such that 
\begin{equation}\label{observability_assumption}
C_T \Vert \phi(0)\Vert_X^2\leq \int_0^T \Vert B^{1/2} \phi(t) \Vert_X^2 \,dt, 
\end{equation}
for every solution of $\phi'(t)+A\phi(t)=0$ (observability inequality for the linear conservative equation).

Then, for every $u_0\in X$, there exists a unique solution 
$u(\cdot)\in C^0(0,+\infty;X)\cap C^1(0,+\infty;D(A)')$
of \eqref{main_eq} such that $u(0)=u_0$.\footnote{Here, the solution is understood in the weak sense, see \cite{CazenaveHaraux,EngelNagel}, and $D(A)'$ is the dual of $D(A)$ with respect to the pivot space $X$. If $u_0\in D(A)$, then 
$u(\cdot)\in C^0(0,+\infty;D(A))\cap C^1(0,+\infty;X)$.}
Moreover, the energy of any solution satisfies
\begin{equation}\label{decayE_cont}
E_u(t) \lesssim T \max(\gamma_1,E_u(0)) \, L\left( \frac{1}{\psi^{-1}( \gamma_2 t )} \right) ,
\end{equation}
for every time $t\geq 0$, with $\gamma_1 \simeq \Vert B\Vert /\gamma_2$ and $\gamma_2 \simeq C_T/ ( T^3\Vert B^{1/2}\Vert^4+T)$.
If moreover 
\begin{equation}\label{condlimsup}
\limsup_{s\searrow 0}\Lambda_H(s)<1,
\end{equation}
then we have the simplified decay rate 
\begin{equation*}
E_{u}(t)\lesssim T \max(\gamma_1, E_u(0)) \, (H')^{-1}\left(\frac{\gamma_3}{t}\right),
\end{equation*}
for every time $t>0$, for some positive constant $\gamma_3\simeq 1$.
\end{theorem}

Theorem \ref{thm_continuous} improves and generalizes to a wide class of equations the main result of \cite{alabau_ammari} in which the authors dealt with locally damped wave equations.  The case of boundary dampings is also treated in \cite{alabauJDE} (see also \cite{alabau_AMO2005}) by a direct method, which provides the same energy decay rates.
The result gives the sharp and general decay rate $L(1/\psi^{-1}(t))$ of the energy at infinity (forgetting about the constants), and the simplified decay rate $(H')^{-1}(1/t)$ under the condition \eqref{condlimsup}. It is also proved in \cite{alabauJDE} that, under this condition, the resulting decay rate is optimal in the finite dimensional case, and for semi-discretized nonlinear wave or plate equations. Hence our estimates are sharp and they are expected to be optimal in the infinite dimensional case. The proof of optimality relies on the derivation of a one-step decay formula for general damping nonlinearity, on a lower estimate based on an energy comparison principle, and on a comparison lemma between time pointwise estimates (such our upper estimate) and lower estimates which are of energy type. Note also that $\rho$ has a linear growth when $g'(0) \neq 0$. In this case the energy decays exponentially at infinity. Moreover, even in the finite dimensional case, optimality cannot be expected when $\displaystyle \limsup_{s\searrow 0}\Lambda_H(s)=1$ (functions $\rho$ leading to that condition are close to a linear growth in a neighborhood of $0$), and it is possible to design examples (linear feedback case) with two branches of solutions that decay exponentially at infinity but do not have the same asymptotic behavior. 

Let us recall some previous well-known results of the literature. The first examples of nonlinear feedbacks were only concerning feedback functions having a polynomial growth in a neighborhood of $0$ (see e.g. \cite{Nakao, Komornik} and the references therein). As far as we know, the first paper considering the case of arbitrary growing feedbacks (in a neighborhood of $0$) is \cite{LT1993}. In this paper, the analysis is based on the existence (always true) of a concave function $h$ satisfying $h(s\rho(s)) \geqslant s^2 + \rho^2(s)$ for all $ |s| \leqslant N$ (see (1.3) in \cite{LT1993}). The paper is very interesting but provides only two examples of construction of such function $h$ in Corollary 2, namely the linear and polynomial growing feedbacks. The results use only the Jensen's inequality (not the Young's inequality), and allow the authors to compare the decay of the energy with the decay of the solution of an ordinary differential equation $S'(t)+q(S(t))=0$ where $q(x)=x - (I+p)^{-1}(x)$ and $p(x)= (cI + h(Cx))^{-1}(Kx)$ where $c, C, K$ are non explicit constants and $f^{-1}$ stands for the inverse function of $f$.  In the general case, these results do not give the ways to build an explicite concave function satisfying  $h(s\rho(s)) \geqslant s^2 + \rho^2(s)$. No general energy decay rates are given in an explicit, simple and general formula, which besides this, could be shown to be "optimal".
Due to this lack of explicit examples of decay rates for arbitrary growing feedbacks in other situations than the linear or polynomial cases, other results were obtained, also based on convexity arguments but through other constructions in \cite{Martinez1, Martinez2} (see also \cite{Nicaise}) through linear energy integral inequalities and in \cite{LiuZuazua} through the comparison with a dissipative ordinary differential inequality. In both cases, optimality is not guaranteed. In particular, \cite{Martinez1, Martinez2}  do not allow to recover the well-known expected "optimal" energy decay rates in the case of polynomially growing feedbacks. Optimality can be shown in particular geometrical situations, in one dimension when the feedback is very weak (as for $\rho(s)=e^{-1/s}$ for $s>0$ close to $0$ for instance), see e.g. \cite{VancostenobleMartinez, alabau_AMO2005}. Hence the challenging questions are not only to derive energy decay rates for arbitrary growing feedbacks, but to determine whether if these decay rates are optimal, at least in finite dimensions and in some situations in the infinite dimensional case, and also to derive one-step, simple and semi-explicite formula which are valid in the general case. This is the main contribution of \cite{alabau_AMO2005, alabauJDE} for direct methods and of the present paper for indirect methods for the continuous as well as the discretized settings (see also \cite{alabau_ammari} for the continuous setting). Note also that the direct method is valid for bounded as well as unbounded feedback operators.
 
Several examples of functions $g$ (see \eqref{ineqg}) are given in Table \ref{table1}, with the corresponding $\Lambda_H$ and decay rates. Note that other examples can be easily built, since the optimal-weight convexity method gives a general and somehow simple way to derive quasi-optimal energy decay rates.

\begin{table}[h]
\begin{center}
\begin{tabular}{|c|c|c|}
\hline
$g(s)$ & $\Lambda_H(s)$ & decay of $E(t)$ \\
\hline\hline
$s / \ln^{-p}(1/s)$, $p>0$ & $\displaystyle \limsup_{x\searrow 0}\Lambda_H(s)=1$ & $e^{-t^{1/(p+1)}}/t^{1/(p+1)}$ \\
\hline
$s^p$ on $[0,s_0^2]$, $p>1$ & $\Lambda_H(s)\equiv \frac{2}{p+1}<1$ & $t^{-2/(p-1)}$ \\
\hline
$e^{-1/s^2}$ & ${\displaystyle \lim_{s\searrow 0}\Lambda_H(s)}=0$ & $1/\ln(t)$ \\
\hline
$s^p\ln^q(1/s)$, $p>1$, $q>0$ & ${\displaystyle \lim_{s\searrow 0}\Lambda_H(s)}=\frac{2}{p+1}<1$ & $t^{-2/(p-1)}\ln^{-2q/(p-1)}(t)$ \\
\hline
$e^{-\ln^p(1/s)}$, $p>2$ & ${\displaystyle \lim_{s\searrow 0}\Lambda_H(s)}=0$ & $e^{-2\ln^{1/p}(t)}$ \\
\hline
\end{tabular}
\end{center}
\caption{Examples}\label{table1}
\end{table}

In the four last examples listed in Table \ref{table1}, the resulting decay rates are optimal in finite dimension and for the semi-discretized wave and plate equations. Moreover, \eqref{condlimsup} is satisfied.

\begin{remark}
It also is natural to wonder whether the linear feedback case, that is $g(s)=s$ for every $s$ close to $0$, as well as the nonlinear feedback cases which have a linear growth close to $0$, such as the example $g(s)=\arctan{s}$ for $s$ close to $0$, are covered by our approach in a "natural continuous" way. These cases can be treated under a common assumption, which is indeed $g'(0) \neq 0$. Notice that an apparent difficulty lies in the fact that the function $H$ is the identity function and is not strictly convex anymore, so that our construction may seem to fail. This limit case is investigated in the following result, whose proof is postponed at the end of Section \ref{sec:proofthm_continuous}. It is obtained under slight modifications in the proof of Theorem \ref{thm_continuous}. This approach is also valid for the direct approach presented in \cite{alabau_AMO2005, alabauJDE}, which leads to continuous nonlinear integral inequalities. We will also formulate a general result in this direction below.
\end{remark}

\begin{corollary}\label{cor:thmconti}
Let us assume that $g'(0) \neq 0$. Under the same assumptions as those of Theorem \ref{thm_continuous} (namely, the ones at the beginning of Section \ref{sec:mainResConti} as well as the observability inequality \eqref{observability_assumption}), there holds $\limsup_{s\searrow 0} \Lambda_H(s)=1$ and
$$
E_{u}(t)\lesssim T \max(\gamma_1, E_u(0)) \, \exp \left(-\gamma_2t\right),
$$
for every time $t>0$ with $\gamma_1 \simeq \Vert B\Vert /\gamma_2$ and $\gamma_2 \simeq C_T/ ( T^3\Vert B^{1/2}\Vert^4+T)$.
\end{corollary}

Let us now apply this generalization to the direct optimal-weight convexity method as introduced in \cite{alabau_AMO2005, alabauJDE}. In all the results
presented in these two papers, and when the bounded as well as unbounded feedback operator $\rho$ satisfies \eqref{ineqg} with a function $g$
such that $g(0)=0=g'(0)$ we can extend the given proofs to the cases for which $g(0)=0$ whereas $g'(0) \neq 0$ (i.e. when $g$ has a linear growth close to $0$). More precisely, we can extend the proof of Theorems 4.8, 4.9, 4.10 and 4.11 in \cite{alabauJDE}  (but also in
the more general framework as presented in Theorem 4.1) to the case for which $g'(0) \neq 0$. For this, it is sufficient as for the proof of Corollary~\ref{cor:thmconti} to replace $g$ by the sequence of functions
$g_{\varepsilon}$ defined by $g_{\varepsilon}(s)=s^{1+\varepsilon}$ for every $|s| \leqslant 1$, where $\varepsilon \in (0,1)$. One can then apply the optimal-weight convexity method to this sequence, and define the associate optimal-weight function $w_{\varepsilon}(\cdot)=L^{-1}(\dfrac{.}{2\beta})$ in a suitable interval (see the above references for more details). We then prove that
$$
\int_S^T w_{\varepsilon}(E(t))E(t)dt \leqslant M E(S) \quad \forall \ 0 \leqslant S \leqslant T,
$$
where $M$, $\beta$ can be chosen independently on $\varepsilon$. We then let $\varepsilon$ goes to $0$. Thanks to the proof of Corollary~\ref{cor:thmconti}, we know that the sequence $w_{\varepsilon}$ converges poinwise on $(0,\beta r_0^2)$ towards $1$. This leads to the inequality
$$
\int_S^T E(t)dt \leqslant M E(S) \quad \forall \ 0 \leqslant S \leqslant T,
$$ 
from which we deduce that $E$ decays exponentially at infinity (see e.g. \cite{Komornik}).

We next provide some typical examples of situations covered by Theorem \ref{thm_continuous}.

\subsection{Examples}\label{sec_examples}
%
%

\subsubsection{Schr\"odinger equation with nonlinear damping}\label{sec:ex:schro}
Our first typical example is the Schr\"odinger equation with nonlinear damping (nonlinear absorption)
$$
i\partial_t u(t,x) + \triangle u(t,x) + i b(x) u(t,x)\rho(x,\vert u(t,x)\vert)=0,
$$
in a Lipschitz bounded subset $\Omega$ of $\R^n$, with Dirichlet boundary conditions.
In that case, we have $X=L^2(\Omega,\C)$, and the operator $A=-i\triangle$ is the Schr\"odinger operator defined on $D(A)=H^1_0(\Omega,\C)$. The operator $B$ is defined by $(Bu)(x) = b(x)u(x)$ where $b\in L^\infty(\Omega,\R)$ is a nontrivial nonnegative function, and the mapping $F$ is defined by $(F(u))(x) = u(x)\rho(x,\vert u(x)\vert)$, where $\rho$ is a real-valued continuous function defined on $\bar\Omega\times[0,+\infty)$ such that $\rho(\cdot,0)=0$ on $\Omega$, $\rho(x,s)\geq 0$ on $\bar\Omega\times[0,+\infty)$,
and such that there exist a function $g\in C^1([-1,1],\R)$ satisfying all assumptions listed in Section \ref{sec:mainResConti}, and constants $c_1>0$ and $c_2>0$ such that
\begin{equation*}
\begin{split}
c_1g(s)\leq s\rho(x,s)&\leq c_2 g^{-1}(s)\quad\,\textrm{if}\  0\leq s\leq 1,\\
c_1s\leq \rho(x,s)&\leq c_2 s \qquad\qquad \textrm{if}\  s\geq 1,
\end{split}
\end{equation*}
for every $x\in\Omega$.
Here, the energy of a solution $u$ is given by $E_u(t)=\frac{1}{2}\int_\Omega \vert u(t,x)\vert^2\, dx$, and we have $E_u'(t) = - \int_\Omega b(x) \vert u(t,x)\vert^2 \rho(x,\vert u(t,x)\vert)\, dx \leq 0$.
Note that, in nonlinear optics, the energy $E_u(t)$ is called the \textit{power} of $u$.

As concerns the observability assumption \eqref{observability_assumption}, it is well known that, if $b(\cdot)\geq\alpha>0$ on some open subset $\omega$ of $\Omega$, and if there exists $T$ such that the pair $(\omega,T)$ satisfies the \textit{Geometric Control Condition}, then there exists $C_T>0$ such that
\begin{equation*}
C_T \Vert \phi(0,\cdot)\Vert_{L^2(\Omega,\C)}^2\leq \int_0^T\int_\Omega b(x)\vert \phi(t,x)\vert^2 \,dx dt,
\end{equation*}
for every solution $\phi$ of the linear conservative equation $\partial_t\phi-i\triangle\phi=0$ with Dirichlet boundary conditions (see \cite{lebeau}).

\subsubsection{Wave equation with nonlinear damping}\label{sec_nonlinearwave}
We consider the wave equation with nonlinear damping
$$
\partial_{tt}u(t,x) - \triangle u(t,x) + b(x) \rho(x,\partial_t u(t,x))=0 ,
$$
in a $C^2$ bounded subset $\Omega$ of $\R^n$, with Dirichlet boundary conditions.
This equation can be written as a first-order equation of the form \eqref{main_eq}, with $X=H^1_0(\Omega)\times L^2(\Omega)$ and
$$
A = \begin{pmatrix}
0 & -\mathrm{id} \\ -\triangle & 0
\end{pmatrix}
$$
defined on $D(A)=H^1_0(\Omega)\cap H^2(\Omega)\times H^1_0(\Omega)$.
The operator $B$ is defined by $(B(u,v))(x) = (0,b(x)v(x))^\top$ where $b\in L^\infty(\Omega)$ is a nontrivial nonnegative function, and the mapping $F$ is defined by $(F(u,v))(x) = (0,\rho(x,v(x)))^\top$, where $\rho$ is a real-valued continuous function defined on $\bar\Omega\times\R$ such that $\rho(\cdot,0)=0$ on $\Omega$, 
$s\rho(x,s)\geq 0$ on $\bar\Omega\times\R$,
and such that there exist a function $g\in C^1([-1,1],\R)$ satisfying all assumptions listed in Section \ref{sec:mainResConti}, and constants $c_1>0$ and $c_2>0$ such that
\begin{equation}\label{hyprho}
\begin{split}
c_1g(|s|)\leq |\rho(x,s)|&\leq c_2 g^{-1}(|s|)\quad\,\textrm{if}\  |s|\leq 1,\\
c_1|s|\leq |\rho(x,s)|&\leq c_2|s| \qquad\qquad \textrm{if}\  s\geq 1,
\end{split}
\end{equation}
for every $x\in\Omega$.
Here, the energy of a solution $u$ is given by $E_u(t)=\frac{1}{2}\int_\Omega \left( (|\nabla u(t,x)|^2+ (\partial_t u(t,x))^2 \right) dx$, and we have 
$E_u'(t) = - \int_\Omega b(x) \partial_t u(t,x) \rho(x,\partial_t u(t,x))\, dx \leq 0.$

The framework of this example is the one of \cite{alabau_ammari}.

As concerns the observability assumption \eqref{observability_assumption}, it is well known that, if $b(\cdot)\geq\alpha>0$ on some open subset $\omega$ of $\Omega$, and if there exists $T$ such that the pair $(\omega,T)$ satisfies the \textit{Geometric Control Condition}, then there exists $C_T>0$ such that
\begin{equation*}
C_T \Vert (\phi(0,\cdot),\partial_t\phi(0,\cdot))\Vert_{H^1_0(\Omega)\times L^2(\Omega)}^2\leq \int_0^T\int_\Omega b(x) (\partial_t\phi(t,x))^2 \,dx dt,
\end{equation*}
for every solution $\phi$ of the linear conservative equation $\partial_{tt}\phi-\triangle\phi=0$ with Dirichlet boundary conditions (see \cite{BLR}).

\subsubsection{Plate equation with nonlinear damping}\label{sec_explate}
We consider the nonlinear plate equation
$$
\partial_{tt}u(t,x) + \triangle^2 u(t,x) + b(x) \rho(x,\partial_t u(t,x))=0 ,
$$
in a $C^4$ bounded subset $\Omega$ of $\R^n$, with Dirichlet and Neumann boundary conditions.
This equation can be written as a first-order equation of the form \eqref{main_eq}, with $X=H^2_0(\Omega)\times L^2(\Omega)$ and
$$
A = \begin{pmatrix}
0 & -\mathrm{id} \\ \triangle^2 & 0
\end{pmatrix}
$$
defined on $D(A)=\left(H^2_0(\Omega)\cap H^4(\Omega)\right)\times H^2_0(\Omega)$.
The operator $B$ is defined by $(B(u,v))(x) = (0,b(x)v(x))^\top$ where $b\in L^\infty(\Omega)$ is a nontrivial nonnegative function, and the mapping $F$ is defined by $(F(u,v))(x) = (0,\rho(x,v(x)))^\top$, where $\rho$ is a real-valued continuous function defined on $\bar\Omega\times[0,+\infty)$ such that $\rho(\cdot,0)=0$ on $\Omega$, 
$s\rho(x,s)\geq 0$ on $\bar\Omega\times\R$,
and such that there exist a function $g\in C^1([-1,1],\R)$ satisfying all assumptions listed in Section \ref{sec:mainResConti}, and constants $c_1>0$ and $c_2>0$ such that \eqref{hyprho} holds.
Here, the energy of a solution $u$ is given by $E_u(t)=\frac{1}{2}\int_\Omega \left( (\triangle u(t,x))^2 + (\partial_t u(t,x))^2 \right) dx$, and we have 
$E_u'(t) = - \int_\Omega b(x) \partial_t u(t,x) \rho(x,\partial_t u(t,x))\, dx \leq 0.$

The framework of this example is the one of \cite{alabaujee2006}.

A sufficient condition obtained in \cite{lebeau}, ensuring the observability assumption \eqref{observability_assumption}, is the following: if $b(\cdot)\geq\alpha>0$ on some open subset $\omega$ of $\Omega$ for which there exists $T$ such that the pair $(\omega,T)$ satisfies the \textit{Geometric Control Condition}, then there exists $C_T>0$ such that
\begin{equation*}
C_T \Vert (\phi(0,\cdot),\partial_t\phi(0,\cdot))\Vert_{H^2_0(\Omega)\times L^2(\Omega)}^2\leq \int_0^T\int_\Omega b(x) (\partial_t\phi(t,x))^2 \,dx dt,
\end{equation*}
for every solution $\phi$ of the linear conservative equation $\partial_{tt}\phi+\triangle^2\phi=0$ associated to the corresponding boundary conditions.

\subsubsection{Transport equation with nonlinear damping}\label{sec:extransport}
We consider the one-dimensional transport equation
$$
\partial_t u(t,x) + \partial_x u(t,x) + b(x)\rho(x,u(t,x)) = 0,\qquad x\in(0,1),
$$
with periodicity conditions $u(t,0)=u(t,1)$. This equation can be written as a first-order equation of the form \eqref{main_eq}, with $X=L^2(0,1)$ and $A=\partial_x$ defined on $D(A)=\{ u\in H^1(0,1)\ \mid\ u(0)=u(1) \}$. We make on $\rho$ the same assumptions as before. The energy is given by $E_u(t) = \frac{1}{2} \int_0^1 u(t,x)^2\, dx$, and we have 
$E_u'(t) = - \int_\Omega b(x) u(t,x) \rho(x,u(t,x))\, dx \leq 0.$
The observability inequality for the conservative equation is satisfied as soon as the observability time is chosen large enough.

On this example, we note two things.

First of all, the above example can be easily extended in multi-D on the torus $\mathbb{T}^n=\mathbb{R}^n/\mathbb{Z}^n$, by considering the following non-linear transport equation
$$
\partial_t u(t,x) + \mathrm{div}(v(x) u(t,x)) + b(x)\rho(x,u(t,x)) = 0,\qquad t>0, \ x\in \mathbb{T}^n,
$$
where $v$ is a regular vector field on $\mathbb{T}^n$ such that $\mathrm{div}(v)=0$. The divergence-free condition on the function $v$ ensures that the operator $A$ defined by $Az=\mathrm{div}(v(x)z(x))$ on 
$$
D(A)=\{z\in H^1(\mathbb{T}^n)\ \mid \ z(\cdot+e_i)=z(\cdot), \ \forall i\in \llbracket 1,n\rrbracket\},
$$ 
where $e_i$ denotes the $i$-th vector of the canonical basis of $\mathbb{R}^n$, is skew-adjoint.

Second, in 1D we can drop the assumption of zero divergence, by using a simple change of variable, which goes as follows. We consider the equation
$$
\partial_t u(t,x) + v(x)\partial_x u(t,x) + b\rho(x,u(t,x)) = 0,\qquad t>0, \ x\in(0,1),
$$
with $v$ a measurable function on $(0,1)$ such that $0 <v_-\leq v(x)\leq v_+$. Then, using the change of variable $x\mapsto \int_0^x \frac{ds}{v(s)}$, we immediately reduce this equation to the case where $v=1$. 

Hence our results can as well be applied to those cases.

\subsubsection{Dissipative equations with nonlocal terms}

In the three previous examples, the term $b(\cdot)\rho(\cdot,\cdot)$ is a viscous damping which is \textit{local}. In other words, the value at $x$ of the function $F(u)$ does only depend on the value at $x$ of the function $u$.
To illustrate the the potential of our approach and the large family of nonlinearities that it covers, We slightly modify here the examples presented in the sections \ref{sec:ex:schro}, \ref{sec_nonlinearwave}, \ref{sec_explate} and \ref{sec:extransport}, by providing several examples of viscous damping terms containing a non-local term.

We refer to the previous sections for the precisions on the boundary conditions and the functional setting associated to each system. 

We consider the non-linear systems
\begin{eqnarray*}\label{eq:nonloc}
& & i\partial_t u(t,x) + \triangle u(t,x) + i b(x) u(t,x)\rho(\vert u\vert)(t,x)=0\\
& & \partial_{tt}u(t,x) - \triangle u(t,x) + b(x) \rho(\partial_t u)(t,x)=0\\
& & \partial_{tt}u(t,x) + \triangle^2 u(t,x) + b(x) \rho(\partial_t u)(t,x)=0\\
& & \partial_t u(t,x) + \partial_x u (t,x) +b(x)\rho(u)(t,x)=0,
\end{eqnarray*}
where the non-linear term $\rho$ is defined by
$$
\rho(f)(x)= \varphi(f(x),\mathcal{N}(f)),
$$
where $\varphi:\R^2\rightarrow \R$ is a continuous function and $\mathcal{N}:L^2(\Omega)\rightarrow \R$ stands for a non-local term. We can typically choose
$\mathcal{N}(f)=\int_\Omega \chi(x)f(x)\, dx$ with $\chi\in L^2(\Omega)$ whenever $\Omega$ is bounded or $\chi$ smooth with compact support else. In the framework of Section \ref{sec:extransport}, one is also allowed to choose $\mathcal{N}(f)=K\star f$ with $K\in L^2(\mathbb{T}^n)$. 
We also impose that $\varphi$ satisfies the following uniform Lipschitz property: there exists $C>0$ such that
$$
|\varphi(s,\tau)-\varphi(s',\tau)|+|\varphi(s,\tau)-\varphi(s,\tau')|\leq C(|s-s'|+|s|.|\tau-\tau'|)
$$ 
for every $(s,s',\tau,\tau')\in \R^4$. As a consequence, one easily infers that the mapping $\rho$ is Lipschitz from $L^2(\Omega)$ into $L^2(\Omega)$. 

Moreover, we choose the function $\varphi$ odd with respect to its first variable and such that $\varphi(s,\tau)\geq 0$ for every $s\geq 0$ and $\tau\in \R$. It follows that the assumption $f\rho(f)\geq 0$ is satisfied by every $f\in L^2(\Omega)$. 

Finally, we assume that the assumption \eqref{ineqg} is satisfied.

Let us provide an example of such a function $\varphi$. Notice that, if there exist two positive constants $k_1$ and $k_2$ such that for every $\tau\in \R$, there exist two positive real numbers $c_\tau$ and $C_\tau$ in $[k_1,k_2]$ such that
$$
\varphi(s,\tau)\sim c_\tau s^3\quad \textrm{as }s\to 0\qquad \textrm{and}\qquad \varphi(s,\tau)\sim C_\tau s\quad \textrm{as }s\to +\infty ,
$$
then the assumption \eqref{ineqg} is satisfied. A possible function $\varphi$ is given by 
$$
\varphi(s,\tau)=\varphi_1(s)\varphi_2(\tau)\qquad \textrm{where}\qquad \varphi_1:\R\ni s\mapsto s-\sin s
$$
and $\varphi_2:\R\rightarrow \R$ denotes any function bounded above and below by some positive constants, for instance $\varphi_2(\tau)=\pi+\arctan(\tau)$.

\subsection{Proof of Theorem \ref{thm_continuous}}\label{sec:proofthm_continuous}
First of all, note that the global well-posedness follows from usual a priori arguments. Indeed, in in sequel we are going to consider the solution, as long as it is well defined, and establish energy estimates. Since we prove that the energy (which is the Hilbert norm of $X$) is decreasing, the global existence of weak and then strong solutions follows (see, e.g., \cite[Theorem 4.3.4 and Proposition 4.3.9]{CazenaveHaraux}).
Hence, in the sequel, without taking care, we do as if the solution were globally well defined.

Note that uniqueness follows from the assumption that $F$ is locally Lipschitz on bounded sets.

The proof goes in four steps.

\paragraph{First step.} \emph{Comparison of the nonlinear equation with the linear damped model.}

In this first step, we are going to compare the nonlinear equation \eqref{main_eq} with its linear damped counterpart
\begin{equation}\label{cont_lineardamped}
z'(t)+Az(t)+Bz(t)=0.
\end{equation}

\begin{lemma}\label{cont_lem1}
For every solution $u(\cdot)$ of \eqref{main_eq}, the solution of \eqref{cont_lineardamped} such that $z(0)=u(0)$ satisfies
\begin{equation}\label{est_cont_lem1}
\int_0^T \Vert B^{1/2} z(t)\Vert_X^2\, dt \leq 2 \int_0^T \left( \Vert B^{1/2} u(t)\Vert_X^2 + \Vert B^{1/2} F(u(t))\Vert_X^2 \right) dt.
\end{equation}
\end{lemma}

\begin{proof}
Setting $\psi(t)=u(t)-z(t)$, we have
$$
\langle \psi'(t)+A\psi(t)+BF(u(t))-Bz(t) , \psi(t) \rangle_X = 0.
$$
Denoting $E_\psi(t)=\frac{1}{2}\Vert\psi(t)\Vert_X^2$, it follows that
$$
E_\psi'(t) + \Vert B^{1/2} z(t)\Vert_X^2  = - \langle u(t), B F(u(t))\rangle_X + \langle B^{1/2} F(u(t)), B^{1/2} z(t)\rangle_X + \langle B^{1/2} u(t), B^{1/2} z(t)\rangle_X.
$$
Using \eqref{ineqBF}, we have $\langle u(t), B F(u(t))\rangle_X \geq 0$, and hence
$$
E_\psi'(t) + \Vert B^{1/2} z(t)\Vert_X^2  \leq \Vert B^{1/2} F(u(t))\Vert_X \Vert B^{1/2} z(t)\Vert_X + \Vert B^{1/2} u(t)\Vert_X \Vert B^{1/2} z(t)\Vert_X.
$$
Using the Young inequality $ab\leq\frac{a^2}{2\theta}+\theta\frac{b^2}{2}$ with $\theta=\frac{1}{2}$, we get
$$
E_\psi'(t) + \Vert B^{1/2} z(t)\Vert_X^2  \leq \frac{1}{2}\Vert B^{1/2} z(t)\Vert_X^2 + \Vert B^{1/2} F(u(t))\Vert_X^2 + \Vert B^{1/2} u(t)\Vert_X^2,
$$
and thus,
$$
E_\psi'(t) + \frac{1}{2}\Vert B^{1/2} z(t)\Vert_X^2 \leq \Vert B^{1/2} F(u(t))\Vert_X^2 + \Vert B^{1/2} u(t)\Vert_X^2.
$$
Integrating in time, and noting that $E_\psi(0)=0$, we infer that
$$
E_\psi(T) + \frac{1}{2}\int_0^T \Vert B^{1/2} z(t)\Vert_X^2\, dt \leq \int_0^T \left( \Vert B^{1/2} F(u(t))\Vert_X^2 + \Vert B^{1/2} u(t)\Vert_X^2 \right) dt.
$$
Since $E_\psi(T)\geq 0$, the conclusion follows.
\end{proof}

\paragraph{Second step.} \emph{Comparison of the linear damped equation with the conservative linear equation.}

We now consider the conservative linear equation
\begin{equation}\label{cont_conservative}
\phi'(t)+A\phi(t)=0.
\end{equation}

\begin{lemma}\label{cont_lem2}
For every solution $z(\cdot)$ of \eqref{cont_lineardamped}, the solution of \eqref{cont_conservative} such that $\phi(0)=z(0)$ is such that
\begin{equation}\label{est_cont_lem2}
\int_0^T \Vert B^{1/2} \phi(t)\Vert_X^2\, dt \leq k_T \int_0^T \Vert B^{1/2} z(t)\Vert_X^2 dt,
\end{equation}
with $k_T = 8T^2\Vert B^{1/2}\Vert^4 +2$.
\end{lemma}

\begin{proof}
Setting $\theta(t)=\phi(t)-z(t)$, we have
$$
\langle \theta'(t)+A\theta(t)-Bz(t) , \theta(t) \rangle_X = 0.
$$
Denoting $E_\theta(t)=\frac{1}{2}\Vert\theta(t)\Vert_X^2$, it follows that $E_\theta'(t)=\langle Bz(t) , \theta(t) \rangle_X$. Integrating a first time over $[0,t]$, and a second time over $[0,T]$, and noting that $E_\theta(0)=0$, we get
$$
\int_0^T E_\theta(t)\, dt = \int_0^T \int_0^t \langle Bz(s) , \theta(s) \rangle_X \, ds\, dt
= \int_0^T (T-t) \langle Bz(t) , \theta(t) \rangle_X \, dt .
$$
Applying as in the proof of Lemma \ref{cont_lem1} the Young inequality with $\theta=\frac{1}{2}$ yields
$$
\frac{1}{2} \int_0^T \Vert \theta(t)\Vert_X^2\, dt  \leq  \int_0^T T^2\Vert Bz(t)\Vert_X^2\, dt + \frac{1}{4}  \int_0^T \Vert \theta(t)\Vert_X^2\, dt ,
$$
and therefore, since $B$ is bounded,
$$
\frac{1}{4} \int_0^T \Vert \theta(t)\Vert_X^2\, dt  \leq  T^2\Vert B^{1/2}\Vert^2 \int_0^T \Vert B^{1/2}z(t)\Vert_X^2\, dt .
$$
Now, since $\phi(t)=\theta(t)+z(t)$, it follows that
\begin{equation*}
\int_0^T \Vert B^{1/2}\phi(t)\Vert_X^2\, dt  \leq  2 \int_0^T \Vert B^{1/2}\theta(t)\Vert_X^2\, dt + 2 \int_0^T \Vert B^{1/2}z(t)\Vert_X^2\, dt 
\leq k_T \int_0^T \Vert B^{1/2}z(t)\Vert_X^2\, dt  .
\end{equation*}
The lemma is proved.
\end{proof}

\paragraph{Third step.} \emph{Nonlinear energy estimate.}

Let $\beta>0$ (to be chosen large enough, later).
Following the optimal weight convexity method of \cite{alabau_AMO2005, alabauJDE}, we define the function
\begin{equation}\label{def_w}
w(s) = L^{-1}\left( \frac{s}{\beta} \right),
\end{equation}
for every $s\in[0,\beta s_0^2)$.
In the sequel, the function $w$ is a weight in the estimates, instrumental in order to derive our result.

\begin{lemma}\label{cont_lem3}
For every solution $u(\cdot)$ of \eqref{main_eq}, we have
\begin{multline}\label{est_cont_lem3}
\int_0^T w(E_\phi(0)) \left( \Vert B^{1/2}u(t)\Vert_X^2 + \Vert B^{1/2}F(u(t))\Vert_X^2 \right) dt \\
\lesssim  T \Vert B\Vert  H^*(w(E_\phi(0)))  + \left( w(E_{\phi}(0)) + 1 \right) \int_0^T \langle Bu(t),F(u(t)) \rangle_X \, dt   .
\end{multline}
\end{lemma}

\begin{proof}
To prove this inequality, we use the isometric representation of $B$ in the space $L^2(\Omega,\mu)$. Denoting $f=U^{-1}u$, using that $U^{-1}BUf=bf$, $\rho(f)=U^{-1}F(Uf)$, we have, for instance, $\Vert B^{1/2}u\Vert_X^2 = \langle f, bf\rangle_{L^2(\Omega,\mu)}$, and hence it suffices to prove that
\begin{multline}\label{est_cont_lem3_isom}
\int_0^T w(E_\phi(0)) \int_\Omega \left( b f^2 + b\rho(f)^2 \right) d\mu\, dt  \\
\lesssim T \int_\Omega b\, d\mu \  H^*(w(E_\phi(0)))  + \left( w(E_{\phi}(0)) + 1\right) \int_0^T \int_\Omega b f \rho(f)\, d\mu\, dt    .
\end{multline}
Indeed, this implies \eqref{est_cont_lem3} (note that $\int_\Omega b\, d\mu\leq\Vert B\Vert$ by the spectral theorem).

Let us prove \eqref{est_cont_lem3_isom}.
First of all, for every $t\in[0,T]$ we set $\Omega_1^t = \{ x\in\Omega \mid \vert f(t,x)\vert \leq \varepsilon_0 \}$. If $b=0$ on $\Omega_1^t$ then the forthcoming integrals (see in particular the left-hand side of \eqref{lem3_ccl1} are zero and there is nothing to prove; hence, without loss of generality we assume that $b$ is nontrivial on $\Omega_1^t$.
Using \eqref{ineqg}, we choose $\varepsilon_0>0$ small enough such that $\frac{1}{c_2^2}  \rho(f)^2\leq s_0^2$ almost everywhere in $\Omega_1^t$, and therefore we have
$$
\frac{1}{\int_{\Omega_1^t} b\, d\mu} \int_{\Omega_1^t} \frac{1}{c_2^2}\rho(f)^2 b\, d\mu \in [0,s_0^2].
$$
Using the Jensen inequality with the measure $b\, d\mu$, and using the fact that $H(x)=\sqrt{x} g(\sqrt{x})$, we get
\begin{equation*}
H \left( \frac{1}{\int_{\Omega_1^t} b\, d\mu} \int_{\Omega_1^t} \frac{1}{c_2^2}  \rho(f)^2 b\, d\mu \right)
\leq  \frac{1}{\int_{\Omega_1^t} b\, d\mu} \int_{\Omega_1^t} \frac{1}{c_2} \vert \rho(f)\vert \ g\left( \frac{1}{c_2} \vert \rho(f)\vert \right) b\, d\mu  .
\end{equation*}
Using \eqref{ineqg}, we have 
$\vert\rho(f)(x)\vert \leq c_2\, g^{-1}(\vert f(x)\vert)$ for almost every $x\in\Omega_1^t$, and since $g$ is increasing, we get that $g\left( \frac{1}{c_2} \vert \rho(f)\vert \right)\leq\vert f\vert$ almost everywhere in $\Omega_1^t$. Since $f \rho(f) \geq 0$ by \eqref{assumptionfrhof},
we get
\begin{equation*}
H \left( \frac{1}{\int_{\Omega_1^t} b\, d\mu} \int_{\Omega_1^t} \frac{1}{c_2^2}  \rho(f)^2 b\, d\mu \right)
\leq \frac{1}{\int_{\Omega_1^t} b\, d\mu} \frac{1}{c_2} \int_{\Omega_1^t} b \vert f\vert\vert\rho(f)\vert \, d\mu 
\leq \frac{1}{\int_{\Omega_1^t} b\, d\mu} \frac{1}{c_2} \int_{\Omega} b f \rho(f) \, d\mu .
\end{equation*}
Since $H$ is increasing, it follows that
$$
\int_{\Omega_1^t} b \rho(f)^2\, d\mu \leq c_2^2 \int_{\Omega_1^t} b\, d\mu \ H^{-1} \left( \frac{1}{\int_{\Omega_1^t} b\, d\mu} \frac{1}{c_2} \int_{\Omega} b f \rho(f) \, d\mu \right) ,
$$
and therefore,
$$
\int_0^T w(E_{\phi}(0)) \int_{\Omega_1^t} b \rho(f)^2\, d\mu \, dt \leq \int_0^T w(E_{\phi}(0))  c_2^2 \int_{\Omega_1^t} b\, d\mu \ H^{-1} \left( \frac{1}{\int_{\Omega_1^t} b\, d\mu} \frac{1}{c_2} \int_{\Omega} b f \rho(f) \, d\mu \right) dt .
$$
Thanks to the Young inequality $AB\leq H(A)+H^*(B)$ (where $H^*$ is the convex conjugate), we infer that
\begin{equation}\label{lem3_ccl1}
\begin{split}
\int_0^T w(E_{\phi}(0)) \int_{\Omega_1^t} b \rho(f)^2\, d\mu \, dt 
&\leq \int_0^T c_2 \int_\Omega b f \rho(f) \, d\mu + \int_0^T  c_2^2 \int_{\Omega_1^t} b\, d\mu \ H^*(w(E_{\phi}(0))) \, dt \\
&\leq   c_2 \int_0^T \int_\Omega b f \rho(f) \, d\mu\, dt   + c_2^2 T \int_{\Omega} b\, d\mu \ H^*(w(E_{\phi}(0))) .
\end{split}
\end{equation}

Besides, in $\Omega\setminus\Omega_1^t$, using \eqref{ineqg} we have $\vert\rho(f)\vert \lesssim |f|$. Using \eqref{assumptionfrhof}, it follows that
\begin{equation}\label{lem3_ccl2}
\begin{split}
\int_0^T w(E_\phi(0)) \int_{\Omega\setminus\Omega_1^t} b \rho(f)^2\, d\mu\, dt 
& \lesssim  \int_0^T w(E_\phi(0)) \int_{\Omega\setminus\Omega_1^t} b \vert f\vert \vert \rho(f)\vert \, d\mu\, dt  \\
& \lesssim  \int_0^T w(E_\phi(0)) \int_{\Omega} b f \rho(f) \, d\mu\, dt  .
\end{split}
\end{equation}

From \eqref{lem3_ccl1} and \eqref{lem3_ccl2}, we infer that
$$
\int_0^T w(E_\phi(0)) \int_{\Omega} b \rho(f)^2\, d\mu\, dt 
\lesssim \left( w(E_\phi(0))  +1\right) \int_0^T \int_{\Omega} b f \rho(f) \, d\mu\, dt   + T \int_{\Omega} b\, d\mu \ H^*(w(E_{\phi}(0))) .
$$

Let us now proceed in a similar way in order to estimate the term $\int_0^T w(E_\phi(0)) \int_{\Omega} b f^2\, d\mu\, dt $.
We set $r_1^2=H^{-1}\left( \frac{c_1}{c_2}H(s_0^2)\right)$ and $\varepsilon_1=\min(s_0,g(r_1))\leq 1$. For every $t\in[0,T]$, we define $\Omega_2^t=\{x\in\Omega \mid \vert f(t,x)\vert\leq\varepsilon_1\}$. As before, without loss of generality we assume that $b$ is nontrivial on $\Omega_2^t$.
From \eqref{ineqg}, we have $c_1 g(\vert f\vert)\leq \vert\rho(f)\vert$ in $\Omega_2^t$. By construction, we have
$$
\frac{1}{\int_{\Omega_2^t} b\, d\mu} \int_{\Omega_2^t} f^2 b\, d\mu \in [0,s_0^2].
$$
Using the Jensen inequality as previously, and using \eqref{ineqg} and \eqref{assumptionfrhof}, we infer that
\begin{multline*}
H\left( \frac{1}{\int_{\Omega_2^t} b\, d\mu} \int_{\Omega_2^t} f^2 b\, d\mu \right) 
\leq \frac{1}{\int_{\Omega_2^t} b\, d\mu} \int_{\Omega_2^t} \vert f\vert \vert g(f)\vert b\, d\mu \\
\leq \frac{1}{c_1\int_{\Omega_2^t} b\, d\mu} \int_{\Omega_2^t} b \vert f\vert \vert\rho(f)\vert\, d\mu  
\leq \frac{1}{c_1\int_{\Omega_2^t} b\, d\mu} \int_{\Omega} b f\rho(f) \, d\mu \\
\end{multline*}
Since $H$ is increasing, and integrating in time, we get
$$
\int_0^T w(E_\phi(0)) \int_{\Omega_2^t} f^2 b\, d\mu \, dt \leq
\int_0^T w(E_\phi(0)) \int_{\Omega_2^t} b\, d\mu \ H^{-1} \left( \frac{1}{c_1\int_{\Omega_2^t} b\, d\mu} \int_{\Omega} b f\rho(f) \, d\mu  \right) \, dt 
$$
It then follows from the Young inequality that
$$
\int_0^T w(E_\phi(0)) \int_{\Omega_2^t} f^2 b\, d\mu \, dt \leq
T \int_{\Omega} b\, d\mu \ H^*(w(E_\phi(0))) + \frac{1}{c_1} \int_0^T \int_{\Omega} b f\rho(f) \, d\mu \, dt
$$
The estimate in $\Omega\setminus\Omega_2^t$ is obtained in a similar way.

The lemma is proved.
\end{proof}

\paragraph{Fourth step.} \emph{End of the proof.}

\begin{lemma}\label{cont_lem4}
We have
\begin{equation}\label{inegEu}
E_u(T) \leq E_u(0)\left( 1-\rho_T L^{-1}\left(\frac{E_u(0)}{\beta}\right) \right) ,
\end{equation}
for some positive sufficiently small constant $\rho_T$.
\end{lemma}

\begin{proof}
Using successively the observability inequality \eqref{observability_assumption}, the estimate \eqref{est_cont_lem2} of Lemma \ref{cont_lem2} and the estimate \eqref{est_cont_lem1} of Lemma \ref{cont_lem1}, we first get that
\begin{equation*}
\begin{split}
2C_T E_\phi(0) \leq  \int_0^T \Vert B^{1/2} \phi(t)\Vert_X^2\, dt 
& \leq k_T \int_0^T \Vert B^{1/2} z(t)\Vert_X^2 dt \\
& \leq 2 k_T \int_0^T \left( \Vert B^{1/2} u(t)\Vert_X^2 + \Vert B^{1/2} F(u(t))\Vert_X^2 \right) dt .
\end{split}
\end{equation*}
Multiplying this inequality by the constant $w(E_\phi(0))$, it follows from the estimate \eqref{est_cont_lem3} of Lemma \ref{cont_lem3} that
\begin{equation}\label{train9h19}
C_T w(E_\phi(0)) E_\phi(0)
\lesssim  k_T T \Vert B\Vert H^*(w(E_\phi(0)))  + k_T  \left( w(E_{\phi}(0)) +1\right) \int_0^T \langle Bu(t),F(u(t)) \rangle_X \, dt ,
\end{equation}
From \eqref{defL} and \eqref{def_w}, we have $L(w(s))=\frac{s}{\beta} = \frac{H^*(w(s))}{w(s)}$ for every $s\in[0,\beta s_0^2)$, and hence $\beta H^*(w(s)) = s w(s)$.
We choose $\beta$ large enough such that $E_\phi(0) < \beta s_0^2$, and thus in particular we get
\begin{equation}\label{train9h20}
H^*(w(E_\phi(0))) = \frac{w(E_\phi(0))E_\phi(0)}{\beta}.
\end{equation}
Besides, we also choose $\beta$ large enough such that $\beta \geq \frac{E_\phi(0)}{L(H'(s_0^2))}$,
and since $L:[0,+\infty)\rightarrow[0,s_0^2)$ is continuous and increasing, it follows that 
\begin{equation}\label{train9h29}
w(E_\phi(0))\leq H'(s_0^2) \lesssim 1.
\end{equation}
Finally, we get from \eqref{train9h19}, \eqref{train9h20} and \eqref{train9h29} that
\begin{equation*}
\left( C_T- \frac{k_T T\Vert B\Vert}{\beta} \right) w(E_\phi(0)) E_\phi(0)
\lesssim k_T \int_0^T \langle Bu(t),F(u(t)) \rangle_X \, dt ,
\end{equation*}
We choose $\beta$ large enough such that $C_T- \frac{k_T T\Vert B\Vert}{\beta} > \frac{C_T}{2}$.
It follows that
\begin{equation}\label{metz08:46}
\rho_T w(E_\phi(0)) E_\phi(0) \leq  \int_0^T \langle Bu(t),F(u(t)) \rangle_X \, dt =E_u(0)-E_u(T),
\end{equation}
for some positive sufficiently small constant $\rho_T$, and since $E_\phi(0)=E_u(0)$, using \eqref{def_w}, we have obtained that
\begin{equation*}
E_u(T) \leq E_u(0) \left( 1 - \rho_T L^{-1}\left( \frac{E_u(0)}{\beta} \right) \right),
\end{equation*}
as expected.
\end{proof}

By translation invariance, we get from \eqref{inegEu} that
$$
E_u((k+1)T) \leq E_u(kT)\left( 1-\rho_T L^{-1}\left(\frac{E_u(kT)}{\beta}\right) \right) ,
$$
for every integer $k$. Setting $E_k = \frac{E_u(kT)}{\beta}$ and $M(x) = xL^{-1}(x)$ for every $x\in[0,s_0^2]$, we have $E_{k+1}-E_k + \rho_T M(E_k) \leq 0$,
for every integer $k$.
From these inequalities, we deduce by routine arguments (similar to those in \cite{alabau_ammari}, and thus not reproduced here) that, setting $K_r(\tau)=\int_\tau^r \frac{dy}{M(y)}$, we have
$$
M(E_p) \leq \frac{1}{\rho_T} \min_{\ell\in\{0,\ldots,p\}} \left(  \frac{K_r^{-1}(\rho_T(p-\ell)) }{\ell+1} \right) ,
$$
from which it follows that
$$
E_u(t) \leq \beta T M^{-1} \left( \inf_{0<\theta<T-t} \left( \frac{1}{\theta}K_r^{-1}\left( \rho_T\frac{t-T-\theta}{T} \right) \right)  \right) \leq \beta T L\left( \frac{1}{\psi^{-1}\left( \rho_T\frac{t-T}{T} \right)} \right),
$$
for sufficiently large $t$, with $\psi$ defined by \eqref{defpsir}. The precise constants in the estimate \eqref{decayE_cont} follow from the choice of $\beta$ (large enough) above.
Theorem \ref{thm_continuous} is proved.

\paragraph{Proof of Corollary \ref{cor:thmconti}.}
Since $g'(0) \neq 0$ and $g$ is strictly increasing, there exist $c_1>0$ and $c_2>0$ such that 
\begin{equation}\label{lab1}
c_1|s| \leqslant |g(s)| \leqslant c_2|s| \quad \forall \ |s| \leqslant 1.
\end{equation}
To overcome the difficulty of dealing with non strictly convex functions, let us consider for $\varepsilon>0$ small enough, the function $g_\varepsilon(s)=s^{1+\varepsilon}$ for $|s|\leq 1$. According to \eqref{lab1},  we have
$$
c_1\, g_\varepsilon (\vert f(x)\vert) \leq \vert\rho(f)(x)\vert \leq c_2\, g_\varepsilon^{-1}(\vert f(x)\vert) \quad \textrm{for almost every}\ x\in\Omega\ \textrm{such that}\ \vert f(x)\vert\leq 1,    \
$$
Denote respectively by $L_\varepsilon$ and $w_\varepsilon$ the functions defined from $g=g_\varepsilon$ respectively by \eqref{defL} and \eqref{def_w}.

Straightforward computations lead to
$$
L_\varepsilon (r)=\frac{\varepsilon}{2+\varepsilon}\left(\frac{2r}{2+\varepsilon}\right)^{2/\varepsilon},\qquad L^{-1}_\varepsilon (r)=\frac{2+\varepsilon}{2}\left(\frac{(2+\varepsilon)r}{\varepsilon}\right)^{\varepsilon/2},
$$
for $|r|\leq 1$ and it follows that
$$
w_\varepsilon(s) = L^{-1}_\varepsilon\left(\frac{s}{\beta}\right) =\frac{2+\varepsilon}{2}\left(\frac{(2+\varepsilon)s}{\beta \varepsilon}\right)^{\varepsilon/2}
$$
for $s$ small enough. Hence, the family $(w_\varepsilon)_{\varepsilon>0}$ converges pointwisely on $(0, \beta s_0^2)$ to $1$ as $\varepsilon$ tends to 0.

Following the proof of Theorem \ref{thm_continuous}, one shows that the estimate \eqref{metz08:46} is verified with $w_\varepsilon$ as weight function, in other words that
$$
\rho_T w_\varepsilon (E_\phi(0)) E_\phi(0) \leq E_u(0)-E_u(T),
$$
for the same positive constant $\rho_T\in (0,1)$ as the one introduced in \eqref{metz08:46}. Letting $\varepsilon$ go to 0 and using that $ E_\phi(0)=E_u(0)$ leads to 
$$
0< E_u(T)\leq (1-\rho_T)  E_u(0).
$$
It is standard to derive from such an estimate the exponential decay of the energy $E_u$ (see e.g. \cite{Haraux}). For that purpose, let us reproduce the previous reasoning on each interval $[jT,(j+1)T)$ with $j\geq 1$. Using an induction argument, it follows that
$$
0< E_u(jT)\leq (1-\rho_T)^j  E_u(0)=e^{-\omega j}E_u(0),
$$
where $\omega=-\ln ((1-\rho_T))>0$.
Then, for every $t >T$, there exists a unique $j \in \mathbb{N}$ such that $t  \in [jT,(j+1)T)$. Since $j>\frac{t}{T}-1$, and $E$ is nonincreasing, we obtain
$$
0< E_u(t) \leqslant E_u(jT)\leq   e^{-\omega j} E_u(0)=\frac{1}{1-\rho_T}e^{-\omega t/T}E_u(0),
$$ 
whence the expected result.
\section{Discretization issues: uniform decay results}\label{sec_discretization}
\subsection{Semi-discretization in space}\label{sec_semidiscrete}
In this section, we introduce and analyze a general space semi-discrete version of  \eqref{main_eq}. 
Our main objective is to prove a theorem similar to Theorem \ref{thm_continuous}, but in this semi-discrete setting, with estimates that are uniform with respect to the mesh parameter. In order to ensure uniformity, we add an extra numerical viscosity term in our discretization scheme (as in \cite{tebou_zuazua}), and we establish uniform decay rates estimates as those proved in Theorem \ref{thm_continuous}.

Let $\triangle x>0$ be a space discretization parameter, standing for the size of the mesh. We assume that $0<\triangle x<{\triangle x}_0$, for some fixed ${\triangle x}_0>0$.
We now introduce an appropriate setting for semi-discretizations in space, following \cite{LabbeTrelat,LasieckaTriggiani} (see also \cite{ChitourTrelat}). Let $(X_{\triangle x})_{0<\triangle x<{\triangle x}_0}$ be a family of finite-dimensional vector spaces. Without loss of generality, we identify $X_{\triangle x}$ with $\R^{N(\triangle x)}$, where $N(\triangle x)\in\N$. More precisely, if $\{x_{i}\}_{1\leq i\leq N(\triangle x)}$ is a family of mesh nodes and $u$ is a mapping from $X$ to $\R$, then the vector $u_{\triangle x}\in X_{\triangle x}$ is expected to be an approximation of the vector $(u(x_{i}))_{1\leq i\leq N(\triangle x)}$.

Hereafter, the notations $\lesssim$ and $\simeq$, already used in the continuous setting, keep the same meaning as before, with the additional requirement that they also mean that the involved constants are uniform as well with respect to $\triangle x$.

Let us recall what is the Hilbert space usually denoted by $X_{1/2}$ (see, e.g., \cite{EngelNagel}).
Let $\beta$ be a real number belonging to the resolvent of $A$. Then the space $X_{1/2}$ is the image of $X$ under $(\beta\mathrm{id}_X-A)^{-1/2}$, and it is endowed with the norm $\Vert u\Vert_{X_{1/2}} = \Vert (\beta\mathrm{id}_X-A)^{1/2}u\Vert_X$. For instance, when $A^{1/2}$ is well defined, we have $X_{1/2}=D(A^{1/2})$. We set $X_{-1/2}=X_{1/2}'$, where the dual is taken with respect to the pivot space $X$.

We assume that, for every $\triangle x\in(0,{\triangle x}_0)$, there exist linear mappings $P_{\triangle x}:X_{-1/2}\rightarrow X_{\triangle x}$ and $\widetilde{P}_{\triangle x}:X_{\triangle x}\rightarrow X_{1/2}$ such that $P_{\triangle x}\widetilde{P}_{\triangle x}=\mathrm{id}_{X_{\triangle x}}$. 
We assume that the scheme is convergent, that is,
$\Vert (I-\widetilde{P}_{\triangle x}P_{\triangle x})u \Vert_X \rightarrow 0$ as $\triangle x\rightarrow 0$, for every $u\in X$.
Here, we have implicitly used the canonical injections $D(A) \hookrightarrow X_{1/2} \hookrightarrow X \hookrightarrow X_{-1/2}$ (see \cite{EngelNagel}).
Additionally, we assume that $P_{\triangle x}=\widetilde{P}_{\triangle x}^*$. 
Note that, at this step, these assumptions are very general and hold for most of numerical schemes.

For every $\triangle x\in(0,{\triangle x}_0)$, the vector space $X_{\triangle x}$ is endowed with the Euclidean norm $\Vert\ \Vert_{\triangle x}$ defined by $\Vert u_{\triangle x}\Vert_{\triangle x} = \Vert \widetilde{P}_{\triangle x} u_{\triangle x}\Vert_X$, for $u_{\triangle x}\in X_{\triangle x}$. The corresponding scalar product is denoted by $\langle\cdot,\cdot\rangle_{\triangle x}$.
Note that $\Vert \widetilde{P}_{\triangle x}\Vert_{L(X_{\triangle x},X)}=1$ and that, by the Uniform Boundedness Principle, $\Vert P_{\triangle x}\Vert_{L(X,X_{\triangle x})} \lesssim 1$.

For every $\triangle x\in(0,{\triangle x}_0)$, we define the approximation operators $A_{\triangle x}:X_{\triangle x}\rightarrow X_{\triangle x}$ of $A$, and $B_{\triangle x}:X_{\triangle x}\rightarrow X_{\triangle x}$, by $A_{\triangle x}=P_{\triangle x}A\widetilde{P}_{\triangle x}$ (where, in this formula, we have implicitly used the canonical extension of the operator $A:X_{1/2}\rightarrow X_{-1/2}$) and $B_{\triangle x}=P_{\triangle x}B\widetilde{P}_{\triangle x}$. Since $P_{\triangle x}=\widetilde{P}_{\triangle x}^*$, note that $A_{\triangle x}$ is identified with a skew-symmetric matrix, and $B_{\triangle x}$ is identified with a symmetric nonnegative matrix.\footnote{We have $\langle B_{\triangle x}u_{\triangle x},u_{\triangle x}\rangle_{\triangle x} = \langle P_{\triangle x} B \tilde P_{\triangle x} u_{\triangle x},u_{\triangle x}\rangle_{\triangle x} = \langle B \tilde P_{\triangle x} u_{\triangle x}, \tilde P_{\triangle x}u_{\triangle x}\rangle_{X}\geq 0$.}
Finally, we define the (nonlinear) mapping $F_{\triangle x}:X_{\triangle x}\rightarrow X_{\triangle x}$ by $F_{\triangle x}(u_{\triangle x}) = P_{\triangle x} F(\widetilde{P}_{\triangle x} u_{\triangle x})$, for every $u_{\triangle x}\in X_{\triangle x}$.
Note that, by construction, $B_{\triangle x}$ is uniformly bounded with respect to $\triangle x$, and $F_{\triangle x}$ is  Lipschitz continuous on bounded subsets of $X_{\triangle x}$, uniformly with respect to $\triangle x$.

We consider the space semi-discrete approximation of \eqref{main_eq} given by
\begin{equation}\label{waveqSpaceDiscrete}
u_{\triangle x}'(t)+A_{\triangle x}u_{\triangle x}(t)+B_{\triangle x}F_{\triangle x}(u_{\triangle x}(t))+(\triangle x)^\sigma\mathcal{V}_{\triangle x} u_{\triangle x}(t)=0.
\end{equation}
The additional term $(\triangle x)^\sigma\mathcal{V}_{\triangle x}u_{\triangle x}(t)$, with $\sigma>0$, is a \emph{numerical viscosity term} whose role is crucial in order to establish decay estimates that are uniform with respect to $\triangle x$. The role and the design of such a term will be discussed further.
We only assume, throughout, that $\mathcal{V}_{\triangle x}:X_{\triangle x}\rightarrow X_{\triangle x}$ is a positive selfadjoint operator.

Defining the energy of a solution $u_{\triangle x}$ of \eqref{waveqSpaceDiscrete} by
\begin{equation}\label{defEnergyDiscrete}
E_{u_{\triangle x}}(t)=\frac{1}{2}\Vert u_{\triangle x}(t)\Vert_{\triangle x}^2 ,
\end{equation}
we have, as long as the solution is well defined,
\begin{equation}\label{decroissanceEnergyDiscrete}
E_{u_{\triangle x}}'(t) = -\langle u_{\triangle x}(t),B_{\triangle x}F_{\triangle x}(u_{\triangle x}(t))\rangle_{\triangle x} 
-(\triangle x)^\sigma\Vert (\mathcal{V}_{\triangle x})^{1/2} u_{\triangle x}(t) \Vert_{\triangle x}^2 .
\end{equation}
We are going to perform an analysis similar to the one done in Section \ref{sec:Continuous}, but in the space semi-discrete setting, with the objective of deriving sharp decay estimates for  \eqref{waveqSpaceDiscrete}, which are uniform with respect to $\triangle x$.

\subsubsection{Main result}
\paragraph{Assumptions and notations.}
First of all, we assume that
\begin{equation}\label{A1}
\langle u_{\triangle x},B_{\triangle x}F_{\triangle x}(u_{\triangle x})\rangle_{\triangle x} 
+ (\triangle x)^\sigma\Vert (\mathcal{V}_{\triangle x})^{1/2} u_{\triangle x} \Vert_{\triangle x}^2
\geq 0,
\end{equation}
for every $u_{\triangle x}\in X_{\triangle x}$. Using \eqref{decroissanceEnergyDiscrete}, this assumption ensures that the energy $E_{u_{\triangle x}}(t)$ defined by \eqref{defEnergyDiscrete} is nonincreasing.

Moreover, we assume that there exists $s_0>0$ such that
\begin{equation}\label{metz2059}
\sup_{\triangle x\in (0,s_0]}(\triangle x)^{\sigma}\Vert \mathcal{V}_{\triangle x}^{1/2}\Vert ^2_{\mathcal{L}(X_{\triangle x})}<+\infty .
\end{equation}
This assumption will be commented in Remark \ref{rk:visco}.


We keep all notations and assumptions done in Section \ref{sec:mainResConti}.
In order to discretize the mapping $\rho:L^2(\Omega,\mu)\rightarrow L^2(\Omega,\mu)$ defined by $\rho(f)=U^{-1}F(Uf)$, we use as well the approximation spaces $X_{\triangle x}$, as follows. We first map the functional setting of $X$ to $L^2(\Omega,\mu)$ by using the  isometry $U^{-1}: X \rightarrow L^2(\Omega,\mu)$. In Section \ref{sec:mainResConti}, we have defined the operator $\bar A=U^{-1}AU$ on $\mathcal{H}=L^2(\Omega,\mu)$, of domain $\mathcal{H}_1 = D(\bar A) = U^{-1}D(A)$. Accordingly, we set $\mathcal{H}_{1/2} = U^{-1} X_{1/2}$, and we define $\mathcal{H}_{-1/2}$ as the dual of $\mathcal{H}_{1/2}$ with respect to the pivot space $\mathcal{H}$. We have $\mathcal{H}_{-1/2}=U^{-1}X_{-1/2}$, where we keep the same notation $U$ to designate the canonical (isometric) extension $U:\mathcal{H}_{-1/2}\rightarrow X_{-1/2}$.
Now, for every $\triangle x\in (0,\triangle x_0)$, the linear mappings $P_{\triangle x} U:H_{-1/2}\rightarrow X_{\triangle x}$ and $U^{-1}\tilde P_{\triangle x}:X_{\triangle x}\rightarrow \mathcal{H}_{1/2}$ give a space discretization of the Hilbert space $\mathcal{H}=L^2(\Omega,\mu)$ on the finite-dimensional spaces $X_{\triangle x}$.
For every $u_{\triangle x}\in X_{\triangle x}$, we set
$$
\rho_{\triangle x}(u_{\triangle x}) = P_{\triangle x} U \rho ( U^{-1} \widetilde{P}_{\triangle x} u_{\triangle x} )
= P_{\triangle x} F( \widetilde{P}_{\triangle x} u_{\triangle x} )
= F_{\triangle x} ( u_{\triangle x} ),
$$
so that, finally, we have $\rho_{\triangle x}=F_{\triangle x}$. 

Besides, for every $f\in L^2(\Omega,\mu)$, we set
\begin{equation}\label{def_rhotilde}
\tilde\rho_{\triangle x}(f) = U^{-1}\tilde P_{\triangle x} P_{\triangle x} U \rho(f)
= U^{-1}\tilde P_{\triangle x} P_{\triangle x} F(Uf) .
\end{equation}
By definition, we have $\tilde\rho_{\triangle x}(U^{-1}\tilde P_{\triangle x} u_{\triangle x})  = U^{-1}\tilde P_{\triangle x}\, \rho_{\triangle x}(u_{\triangle x})$, for every $u_{\triangle x}\in X_{\triangle x}$.
The mapping $\tilde\rho_{\triangle x}$ is the mapping $\rho$ filtered by the ``sampling operator" $U^{-1}\tilde P_{\triangle x} P_{\triangle x}U = (P_{\triangle x}U)^* P_{\triangle x}U $. By assumption, the latter operator converges pointwise to the identity as $\triangle x\rightarrow  0$, and in many numerical schemes it corresponds to take sampled values of a given function $f$.

We have $\tilde \rho_{\triangle x}(0)=0$, but \eqref{assumptionfrhof} and \eqref{ineqg} are not necessarily satisfied, with $\rho$ replaced with $\tilde \rho_{\triangle x}$.
In the sequel, setting $f_{\triangle x} = U^{-1}\tilde P_{\triangle x} u_{\triangle x}$, we assume that 
\begin{equation}\label{assumptionfrhotildef}
f_{\triangle x} \tilde\rho_{\triangle x}(f_{\triangle x}) \geq 0, 
\end{equation}
and that
\begin{equation}\label{ineqgrhotilde}
\begin{split}
c_1\, g(\vert f_{\triangle x}(x)\vert) &\leq \vert\tilde\rho_{\triangle x}(f_{\triangle x})(x)\vert \leq c_2\, g^{-1}(\vert f_{\triangle x}(x)\vert) \quad \textrm{for a.e.}\ x\in\Omega\ \textrm{such that}\ \vert f_{\triangle x}(x)\vert\leq 1,    \\
c_1\, \vert f_{\triangle x}(x)\vert & \leq \vert\tilde\rho_{\triangle x}(f_{\triangle x})(x)\vert \leq c_2 \, \vert f_{\triangle x}(x)\vert\qquad\quad\ \textrm{for a.e.}\ x\in\Omega\ \textrm{such that}\ \vert f_{\triangle x}(x)\vert\geq 1 ,
\end{split}
\end{equation}
for every $u_{\triangle x}\in X_{\triangle x}$, for every $\triangle x\in(0,\triangle x_0)$.

Note that the additional assumptions \eqref{assumptionfrhotildef} and \eqref{ineqgrhotilde} are valid for many classical numerical schemes, such as finite differences, finite elements, and in more general, for any method based on Lagrange interpolation, in which inequalities or sign conditions are preserved under sampling. But for instance this assumption may fail for spectral methods (global polynomial approximation) in which sign conditions may not be preserved at the nodes of the scheme. The same remark applies to \eqref{A1}.

Note also that assuming \eqref{assumptionfrhotildef} and \eqref{ineqgrhotilde} is weaker than assuming \eqref{assumptionfrhof} and \eqref{ineqg} with $\rho$ replaced with $\tilde \rho_{\triangle x}$, for every $f\in L^2(\Omega,\mu)$. Indeed, the inequalities \eqref{assumptionfrhotildef} and \eqref{ineqgrhotilde} are required to hold only at the nodes of the numerical scheme.

\medskip

The main result of this section is the following.

\begin{theorem}\label{thm_space}
In addition to the above assumptions, we assume that there exist $T>0$, $\sigma>0$ and $C_T>0$ such that 
\begin{equation}\label{ineqObsDiscrete}
C_T E_{\phi_{\triangle x}}(0)\leq \int_0^T\left(\Vert B_{\triangle x}^{1/2}\phi_{\triangle x}(t)\Vert_{\triangle x}^2+(\triangle x)^\sigma\Vert \mathcal{V}_{\triangle x}^{1/2}\phi_{\triangle x}(t)\Vert_{\triangle x}^{2}\right)\, dt,
\end{equation}
for every solution of $\phi_{\triangle x}'(t)+A_{\triangle x}\phi_{\triangle x}(t)=0$ (uniform observability inequality with viscosity for the space semi-discretized linear conservative equation).

Then, the solutions of \eqref{waveqSpaceDiscrete}, with values in $X_{\triangle x}$, are well defined on $[0,+\infty)$, and the energy of any solution satisfies
$$
E_{u_{\triangle x}}(t) \lesssim T \max(\gamma_1, E_{u_{\triangle x}}(0)) \, L\left( \frac{1}{\psi^{-1}( \gamma_2 t )} \right) ,
$$
for every $t> 0$ and $\triangle x\in (0,s_0]$, with $\gamma_1 \simeq \Vert B\Vert /\gamma_2$ and $\gamma_2 \simeq C_T/ T( T^2\Vert B^{1/2}\Vert^4+1)$.
Moreover, under \eqref{condlimsup}, we have the simplified decay rate
$$
E_{u_{\triangle x}}(t) \lesssim T \max(\gamma_1, E_{u_{\triangle x}}(0)) \, (H')^{-1}\left(\frac{\gamma_3}{t}\right),
$$
for every $t> 0$ and $\triangle x\in (0,s_0]$, for some positive constant $\gamma_3\simeq 1$.
\end{theorem}

\begin{remark}[Comments on the uniform space semi-discrete observability inequality \eqref{ineqObsDiscrete}]\label{rk:commentsdiscobsineq}
The main assumption above is the uniform observability inequality \eqref{ineqObsDiscrete}, which is not easy to obtain in general. There are not so many general results in the existing literature, providing such uniform estimates.

\medskip

First of all, in the absence of a viscosity term, the observability inequality \eqref{ineqObsDiscrete} fails to be uniform in general (see \cite{ZuazuaSIREV} and references therein), in the sense that the largest constant appearing at the left-hand side of \eqref{ineqObsDiscrete}, depending on $\triangle x$ in general, tends to $0$ as $\triangle x\rightarrow 0$. This phenomenon, which is by now well known, is due to the fact that the discretization creates spurious highfrequency oscillations that cause a vanishing speed of highfrequency wave packets. We refer to \cite{ZuazuaSIREV} for the detailed description of this lack of uniformity, in particular for 1D wave equations with boundary observation (for which a simple computation yields non-uniformity).

A counterexample to uniformity is provided in \cite{MaricaZuazua} for the 1D Dirichlet wave equation with internal observation (over a subset $\omega$), semi-discretized in space by finite differences: it is proved that, for every solution of the conservative system
$\phi_{\triangle x}''(t) - \triangle_{\triangle x}\phi_{\triangle x}(t) = 0 $, 
there holds
$
C_{\triangle x} E_{\phi_{\triangle x}}(0) \leq  \int_0^T \chi_\omega \Vert \phi_{\triangle x}'(t)\Vert_{\triangle x}^2 \, dt ,
$
where the largest positive constant $C_{\triangle x}$ for which this inequality is valid satisfies $C_{\triangle x}\rightarrow 0$ as $\triangle x\rightarrow 0$.
The proof of this lack of uniformity combines the following facts: using gaussian beams, it is shown that, along every bicharacteristic ray, there exists a solution of the wave equation whose energy is localized along this ray; the velocity of highfrequency wave packets for the discrete model tends to $0$ as $\triangle x$ tends to $0$. Then, for every $T> 0$, for $\triangle x > 0$ small enough, there exist initial data whose corresponding solution is concentrated along a ray that does not reach the observed region $\omega$ within time $T$.

\medskip

Note that the uniform observability inequality \eqref{ineqObsDiscrete} holds for all solutions of the linear conservative equation $\phi_{\triangle x}'(t)+A_{\triangle x}\phi_{\triangle x}(t)=0$, if and only if
the solutions of $z_{\triangle x}'(t)+A_{\triangle x}z_{\triangle x}(t)+B_{\triangle x} z_{\triangle x}(t)=0$ are exponentially decaying, with a uniform exponential rate (see \cite{EZ1}, see also Lemma \ref{lemma2} and the end of Section \ref{proofTheoDiscrSpace}, from which this claim follows by an easy adaptation).

\medskip

Many possible remedies to the lack of uniformity have been proposed in \cite{ZuazuaSIREV}, among which the filtering of highfrequencies, the use of multigrids, or the use of appropriate viscosity terms.
Here, since we are in a nonlinear context, we focus on the use of viscosity, that we find more appropriate.\footnote{Note that, in the linear context, filtering and adding a viscosity term are equivalent, as it follows from \cite[Corollary 3.8 and Remark 3.9]{EZ1}.}
The role of the numerical viscosity term (which vanishes as the mesh size tends to zero) is to damp out the highfrequency numerical spurious oscillations that appear in the semi-discrete setting.

\medskip

A typical choice of the viscosity operator, proposed in \cite{EZ1}, is
$$
\mathcal{V}_{\triangle x} = \sqrt{A_{\triangle x}^*A_{\triangle x}},
$$ 
(symmetric positive definite square root of $A_{\triangle x}$; if it is not positive definite, add some $\varepsilon\,\mathrm{id}_{X_{\triangle x}}$), with $\sigma$ chosen such that \eqref{A1} 
is satisfied. Some variants of the viscosity term are possible.
Unfortunately, it is not proved in the existing literature that such a general viscosity term is systematically sufficient in order to recover the desired uniform properties (in contrast to time discretizations, see further).
As discussed in \cite{EZ1}, the main difficulty consists in establishing the uniform observability inequality \eqref{ineqObsDiscrete} with viscosity (see \cite{ZuazuaSIREV} for a thorough discussion on this issue).

There are quite few results in the literature where one can find such uniform stability results, for some particular classes of equations.

One of them concerns the wave equation, studied in 1D (and in a 2D square) in \cite{tebou_zuazua}, with an internal damping (see also the generalization to any regular 2D domain in \cite{MunchPazoto}), semi-discretized in space by finite differences. Using discrete multipliers, the authors prove that the solutions of the semi-discretized locally damped wave equation with viscosity
$$
y_{\triangle x}''(t)-\triangle_{\triangle x} y_{\triangle x}(t) + a_{\triangle x} y_{\triangle x}'(t) - (\triangle x)^2\triangle_{\triangle x} y_{\triangle x}'(t) = 0  ,
$$
decay exponentially to $0$, uniformly with respect to $\triangle x$. Here, the viscosity term is $- (\triangle x)^2\triangle_{\triangle x} y_{\triangle x}'$, and the term $a_{\triangle x}$ stands for the discretization of a localized damping.
With our notations, we have
$$
u_{\triangle x}(t) = \begin{pmatrix}
y_{\triangle x}(t)\\
y_{\triangle x}'(t)
\end{pmatrix},
\quad
A_{\triangle x} = \begin{pmatrix}
0 & -I_N\\
-\triangle_{\triangle x} & 0 
\end{pmatrix},
\quad
B_{\triangle x} = \begin{pmatrix}
0 & 0\\
0 & a_{\triangle x} 
\end{pmatrix},
\quad
\mathcal{V}_{\triangle x} = \begin{pmatrix}
0 & 0\\
0 & -\triangle_{\triangle x} 
\end{pmatrix} ,
$$
with $\sigma=2$, and with $\triangle_{\triangle x}$, the usual finite-difference discretization of the Laplacian, given for instance in 1D by
$$
\triangle_{\triangle x} = \frac{1}{(\triangle x)^2}\begin{pmatrix}
-2 & 1 & 0 & \hdots & 0\\
1 & \ddots & \ddots & \ddots & \vdots\\
0 & \ddots & \ddots & \ddots & 0\\
\vdots & \ddots & \ddots & \ddots & 1\\
0 & \hdots & 0 & 1 & -2
\end{pmatrix} .
$$
In that case, the viscosity operator is not positive definite, however the uniform observability inequality \eqref{ineqObsDiscrete} follows from the proof of \cite[Theorem 1.1]{EZ1}.

In \cite[Theorem 7.1]{Ervedoza_NM2009}, the author studies a class of second-order equations, with a self-adjoint positive operator, semi-discretized in space by means of finite elements on general meshes, not necessarily regular. This result is generalized in \cite{Miller} with a weaker numerical viscosity term. We refer also to \cite{ANVE} for a similar study under appropriate spectral gap conditions, noting that, in that paper, results are also provided on uniform polynomial stability.

Finally, in \cite{RTT_JMPA2006} and \cite{RTT_COCV2007}, the authors use a similar viscosity operator for the plate equation and even for a more general class of second-order evolution equations (under appropriate spectral gap assumptions, and thus essentially in 1D). They do not explicitly prove a uniform observability inequality but they establish a close result based on a frequential characterization of the dissipation introduced in \cite{Liu-Zheng}.
\end{remark}

\begin{remark}[Comments on the assumption \eqref{metz2059} on the numerical viscosity.]\label{rk:visco}
The assumption \eqref{metz2059} is satisfied in all cases mentioned in Remark \ref{rk:commentsdiscobsineq} (see the corresponding references): in \cite{tebou_zuazua}, some explanations about the choice of the viscosity operator and its properties (including \eqref{metz2059}) are given after Theorem 1.1; in \cite{EZ1}, \eqref{metz2059} corresponds to the third assumption on the viscosity operator in Theorem 3.7; in \cite{MunchPazoto}, a finite difference approximation scheme is used,  with $\sigma=2$ in the viscosity (it is easy to see that \eqref{metz2059} is indeed satisfied); finally, in \cite{RTT_JMPA2006,RTT_COCV2007}, this assumption follows from their particular choice of the numerical viscosity operator. 
\end{remark}

\subsubsection{Proof of Theorem \ref{thm_space}}\label{proofTheoDiscrSpace}
We follow the lines of the proof of Theorem \ref{thm_continuous}. The only difference is in handling the viscosity term. Hereafter, we provide the main steps and we only give details when there are some differences with respect to the continuous case.

Note that the global well-posedness of the solutions follows from usual a priori arguments (using energy estimates), as in the continuous setting. Uniqueness follows as well from the assumption that $F$ is locally Lipschitz on bounded sets.

\paragraph{First step.} \textit{Comparison of the nonlinear equation \eqref{waveqSpaceDiscrete} with the linear damped model.} 

We first compare the nonlinear equation \eqref{waveqSpaceDiscrete} with the linear damped system
\begin{equation}\label{waveqSpaceLin}
z_{\triangle x}'(t)+A_{\triangle x}z_{\triangle x}(t)+B_{\triangle x} z_{\triangle x}(t)+(\triangle x)^\sigma\mathcal{V}_{\triangle x}z_{\triangle x}(t)=0. 
\end{equation}

\begin{lemma}\label{lemma1}
For every solution $u_{\triangle x}(\cdot)$ of \eqref{waveqSpaceDiscrete}, the solution $z_{\triangle x}(\cdot)$ of \eqref{waveqSpaceLin} such that $z_{\triangle x}(0)=u_{\triangle x}(0)$ satisfies
\begin{multline*}
\int_0^T \left(\Vert B_{\triangle x}^{1/2} z_{\triangle x}(t)\Vert_{\triangle x}^2+(\triangle x)^\sigma\Vert \mathcal{V}_{\triangle x}^{1/2}z_{\triangle x}(t)\Vert_{\triangle x}^2\right)\, dt \\
\leq 2\int_0^T\left(\Vert B_{\triangle x}^{1/2} u_{\triangle x}(t)\Vert_{\triangle x}^{2}+\Vert B_{\triangle x}^{1/2}F_{\triangle x}(u_{\triangle x}(t))\Vert_{\triangle x}^{2}+2(\triangle x)^\sigma\Vert \mathcal{V}_{\triangle x}^{1/2}u_{\triangle x}(t)\Vert_{\triangle x}^2\right)\, dt.
\end{multline*}
\end{lemma}

\begin{proof}
Setting $\psi_{\triangle x}(t)=u_{\triangle x}(t)-z_{\triangle x}(t)$, we have
$$
\langle \psi_{\triangle x}'(t)+A_{\triangle x}\psi_{\triangle x}(t)+B_{\triangle x}F_{\triangle x}(u_{\triangle x}(t))-B_{\triangle x}z_{\triangle x}(t)+(\triangle x)^\sigma\mathcal{V}_{\triangle x}\psi_{\triangle x}(t) , \psi_{\triangle x}(t) \rangle_{\triangle x} = 0.
$$
Denoting $E_{\psi_{\triangle x}}(t)=\frac{1}{2}\Vert\psi_{\triangle x}(t)\Vert_{\triangle x}^{2}$, it follows that
\begin{multline*}
E_{\psi_{\triangle x}}'(t) + \Vert B_{\triangle x}^{1/2} z_{\triangle x}(t)\Vert_{\triangle x}^{2}+(\triangle x)^\sigma\Vert \mathcal{V}_{\triangle x}^{1/2}\psi_{\triangle x}\Vert_{\triangle x}^2  
=   - \langle u_{\triangle x}(t), B_{\triangle x} F_{\triangle x}(u_{\triangle x}(t))\rangle_{\triangle x}\\
+ \langle B_{\triangle x}^{1/2} F_{\triangle x}(u_{\triangle x}(t)), B_{\triangle x}^{1/2} z_{\triangle x}(t)\rangle_{\triangle x}+ \langle B_{\triangle x}^{1/2} u_{\triangle x}(t), B_{\triangle x}^{1/2} z_{\triangle x}(t)\rangle_{\triangle x}.
\end{multline*}
Using \eqref{A1}, we have $\langle u_{\triangle x}(t), B_{\triangle x} F_{\triangle x}(u_{\triangle x}(t))\rangle_{\triangle x} \geq -(\triangle x)^\sigma\Vert \mathcal{V}_{\triangle x}^{1/2}u_{\triangle x}(t)\Vert_{\triangle x}^2$, and hence
\begin{multline*}
E_{\psi_{\triangle x}}'(t) + \Vert B_{\triangle x}^{1/2} z_{\triangle x}(t)\Vert_{\triangle x}^{2}+(\triangle x)^\sigma\Vert \mathcal{V}_{\triangle x}^{1/2}\psi_{\triangle x}\Vert_{\triangle x}^2
\leq \Vert B_{\triangle x}^{1/2} F_{\triangle x}(u_{\triangle x}(t))\Vert_{\triangle x} \Vert B_{\triangle x}^{1/2} z_{\triangle x}(t)\Vert_{\triangle x} \\
+ \Vert B_{\triangle x}^{1/2} u_{\triangle x}(t)\Vert_{\triangle x} \Vert B_{\triangle x}^{1/2} z_{\triangle x}(t)\Vert_{\triangle x} + (\triangle x)^\sigma\Vert \mathcal{V}_{\triangle x}^{1/2}u_{\triangle x}(t)\Vert_{\triangle x}^2.
\end{multline*}
Thanks to the Young inequality $ab\leq\frac{a^2}{2\theta}+\theta\frac{b^2}{2}$ with $\theta=\frac{1}{2}$, we get
\begin{multline*}
E_{\psi_{\triangle x}}'(t) + \Vert B_{\triangle x}^{1/2} z_{\triangle x}(t)\Vert_{\triangle x}^{2}+(\triangle x)^\sigma\Vert \mathcal{V}_{\triangle x}^{1/2}\psi_{\triangle x}\Vert_{\triangle x}^2 \\
\leq \frac{1}{2}\Vert B_{\triangle x}^{1/2} z_{\triangle x}(t)\Vert_{\triangle x}^{2} + \Vert B_{\triangle x}^{1/2} F_{\triangle x}(u_{\triangle x}(t))\Vert_{\triangle x}^{2} + \Vert B_{\triangle x}^{1/2} u_{\triangle x}(t)\Vert_{\triangle x}^{2} + (\triangle x)^\sigma\Vert \mathcal{V}_{\triangle x}^{1/2}u_{\triangle x}(t)\Vert_{\triangle x}^2,
\end{multline*}
and thus,
\begin{multline*}
E_{\psi_{\triangle x}}'(t) + \frac{1}{2}\Vert B_{\triangle x}^{1/2} z_{\triangle x}(t)\Vert_{\triangle x}^{2}+(\triangle x)^\sigma\Vert \mathcal{V}_{\triangle x}^{1/2}\psi_{\triangle x}\Vert_{\triangle x}^2  \\
\leq \Vert B_{\triangle x}^{1/2} F_{\triangle x}(u_{\triangle x}(t))\Vert_{\triangle x}^{2} + \Vert B_{\triangle x}^{1/2} u_{\triangle x}(t)\Vert_{\triangle x}^{2} + (\triangle x)^\sigma\Vert \mathcal{V}_{\triangle x}^{1/2}u_{\triangle x}(t)\Vert_{\triangle x}^2 .
\end{multline*}
Integrating in time, noting that $E_{\psi_{\triangle x}}(0)=0$, we infer that
\begin{multline*}
E_{\psi_{\triangle x}}(T) + \frac{1}{2}\int_0^T\left( \Vert B_{\triangle x}^{1/2} z_{\triangle x}(t)\Vert_{\triangle x}^{2}+(\triangle x)^\sigma\Vert \mathcal{V}_{\triangle x}^{1/2}z_{\triangle x}\Vert_{\triangle x}^2\right)\, dt \\
\leq \int_0^T \left( \Vert B_{\triangle x}^{1/2} F_{\triangle x}(u_{\triangle x}(t))\Vert_{\triangle x}^{2} + \Vert B_{\triangle x}^{1/2} u_{\triangle x}(t)\Vert_{\triangle x}^{2} +2(\triangle x)^\sigma\Vert \mathcal{V}_{\triangle x}^{1/2}u_{\triangle x}(t)\Vert_{\triangle x}^2\right) \, dt.
\end{multline*}
Since $E_{\psi_{\triangle x}}(T)\geq 0$, the conclusion follows.
\end{proof}

\paragraph{Second step.} \textit{Comparison of the linear damped equation with the conservative linear equation.}

We consider the space semi-discretized conservative linear system
\begin{equation}\label{conservativePhi}
\phi_{\triangle x}'(t)+A_{\triangle x}\phi_{\triangle x}(t)=0 .
\end{equation}

\begin{lemma}\label{lemma2}
Assume that \eqref{metz2059} holds true.
Then, for every solution $z_{\triangle x}(\cdot)$ of \eqref{waveqSpaceLin}, the solution $\phi_{\triangle x}$ of \eqref{conservativePhi} such that $\phi_{\triangle x}(0)=z_{\triangle x}(0)$ satisfies
\begin{multline*}
\int_0^T\left(\Vert B_{\triangle x}^{1/2}\phi_{\triangle x}(t)\Vert_{\triangle x}^{2} +(\triangle x)^\sigma\Vert \mathcal{V}_{\triangle x} ^{1/2}\phi_{\triangle x}(t)\Vert_{\triangle x}^{2}\right)\, dt
\\
\lesssim k_T \int_0^T\left(\Vert B_{\triangle x}^{1/2}z_{\triangle x}(t)\Vert_{\triangle x}^{2}+(\triangle x)^\sigma\Vert \mathcal{V}_{\triangle x} ^{1/2}z_{\triangle x}(t)\Vert_{\triangle x}^{2}\right)\, dt ,
\end{multline*}
with $k_T = 1+ T^2+ T^2 \Vert B^{1/2}\Vert^2 + T^2 \Vert B^{1/2}\Vert^4$.
\end{lemma}

\begin{proof}
Setting $\theta_{\triangle x}(t)=\phi_{\triangle x}(t)-z_{\triangle x}(t)$ and $E_{\theta_{\triangle x}}(t)=\frac{1}{2}\Vert \theta_{\triangle x}(t)\Vert_{\triangle x}^2$, we have
\begin{equation*}
\langle \theta_{\triangle x}'(t)+A_{\triangle x}\theta_{\triangle x}-B_{\triangle x}z_{\triangle x}(t)-(\triangle x)^\sigma\mathcal{V}_{\triangle x} z_{\triangle x}(t),\theta_{\triangle x}(t)\rangle_{\triangle x}=0,
\end{equation*}
and therefore
$$
E'_{\theta_{\triangle x}}(t)=\langle B_{\triangle x} z_{\triangle x}(t),\theta_{\triangle x}(t)\rangle_{\triangle x}+(\triangle x)^\sigma\langle \mathcal{V}_{\triangle x} z_{\triangle x}(t), \theta_{\triangle x}(t)\rangle_{\triangle x}.
$$
Integrating a first time over $[0,t]$ and a second time over $[0,T]$, and noting that $E_{\theta_{\triangle x}}(0)=0$, we get
\begin{equation*}
\frac{1}{2}\int_0^T \Vert \theta_{\triangle x}(t)\Vert_{\triangle x}^2\, dt = \int_0^T(T-t) \left( \langle B_{\triangle x} z_{\triangle x}(t),\theta_{\triangle x}(t)\rangle_{\triangle x}+(\triangle x)^\sigma\langle \mathcal{V}_{\triangle x} z_{\triangle x}(t), \theta_{\triangle x}(t)\rangle_{\triangle x} \right) \, dt.
\end{equation*}
Applying the Young inequality $ab\leq \frac{a^2\nu}{2}+\frac{b^2}{2\nu }$ for some $\nu>0$ to both terms at the right-hand side of the above equality, we get
\begin{multline}\label{metz1115}
\int_0^T \Vert \theta_{\triangle x}(t)\Vert_{\triangle x}^2\, dt 
\leq
T^2\nu\int_0^T\Vert B_{\triangle x}z_{\triangle x}(t)\Vert_{\triangle x}^{2}\, dt+\frac{1}{\nu}\int_{0}^T\Vert \theta_{\triangle x}(t)\Vert_{\triangle x}^{2}\, dt \\
+ T^2\nu (\triangle x)^\sigma \int_0^T\Vert \mathcal{V}_{\triangle x} ^{1/2}z_{\triangle x}(t)\Vert_{\triangle x}^{2}\, dt+ \frac{(\triangle x)^\sigma}{\nu}\int_{0}^T\Vert \mathcal{V}_{\triangle x}^{1/2}\theta_{\triangle x}(t)\Vert_{\triangle x}^{2}\, dt.
\end{multline}
Using moreover that
$$
\int_0^T \left(\Vert \theta_{\triangle x}(t)\Vert_{\triangle x}^2+(\triangle x)^\sigma\Vert \mathcal{V}_{\triangle x}^{1/2}\theta_{\triangle x}(t)\Vert_{\triangle x}^2\right)\, dt  \leq M_{\triangle x} \int_0^T \Vert \theta_{\triangle x}(t)\Vert_{\triangle x}^2\, dt,
$$
with $M_{\triangle x}=1+(\triangle x)^{\sigma}\Vert \mathcal{V}_{\triangle x}^{1/2}\Vert ^2_{\mathcal{L}(X_{\triangle x})}$, one gets for all $\nu> M_{\triangle x}$
\begin{multline*}
\left(1-\frac{M_{\triangle x}}{\nu} \right)\int_0^T \left(\Vert \theta_{\triangle x}(t)\Vert_{\triangle x}^2+(\triangle x)^\sigma\Vert \mathcal{V}_{\triangle x}^{1/2}\theta_{\triangle x}(t)\Vert_{\triangle x}^2\right)\, dt  
\\
\leq T^2\nu M_{\triangle x}\int_0^T\Vert B_{\triangle x}z_{\triangle x}(t)\Vert_{\triangle x}^{2}\, dt + T^2\nu (\triangle x)^\sigma M_{\triangle x} \int_0^T\Vert \mathcal{V}_{\triangle x} ^{1/2}z_{\triangle x}(t)\Vert_{\triangle x}^{2}\, dt .
\end{multline*}
Let us choose $\nu=2M_{\triangle x}$, the last inequality becomes
\begin{multline*}
\frac{1}{2}\int_0^T \left(\Vert \theta_{\triangle x}(t)\Vert_{\triangle x}^2+(\triangle x)^\sigma\Vert \mathcal{V}_{\triangle x}^{1/2}\theta_{\triangle x}(t)\Vert_{\triangle x}^2\right)\, dt  
\\
\leq 2T^2 M_{\triangle x}^2\left(\int_0^T\Vert B_{\triangle x}z_{\triangle x}(t)\Vert_{\triangle x}^{2}\, dt +  (\triangle x)^\sigma \int_0^T\Vert \mathcal{V}_{\triangle x} ^{1/2}z_{\triangle x}(t)\Vert_{\triangle x}^{2}\, dt\right) .
\end{multline*}
As a result, there holds
\begin{multline*}
\frac{1}{2}\int_0^T \left(\Vert \theta_{\triangle x}(t)\Vert_{\triangle x}^2+(\triangle x)^\sigma\Vert \mathcal{V}_{\triangle x}^{1/2}\theta_{\triangle x}(t)\Vert_{\triangle x}^2\right)\, dt  
\\
\lesssim T^2\Vert B^{1/2}\Vert^2 \int_0^T\Vert B_{\triangle x}^{1/2}z_{\triangle x}(t)\Vert_{\triangle x}^{2}\, dt + T^2 (\triangle x)^\sigma \int_0^T\Vert \mathcal{V}_{\triangle x} ^{1/2}z_{\triangle x}(t)\Vert_{\triangle x}^{2}\, dt ,
\end{multline*}
where, to obtain the latter inequality, we have used that $\Vert B_{\triangle x}\Vert \lesssim \Vert B\Vert$.
Writing $\phi_{\triangle x}=\theta_{\triangle x}+z_{\triangle x}$, we have
\begin{multline*}
\int_0^T\left(\Vert B_{\triangle x}^{1/2}\phi_{\triangle x}(t)\Vert_{\triangle x}^{2}\, dt+(\triangle x)^\sigma\Vert \mathcal{V}_{\triangle x} ^{1/2}\phi_{\triangle x}(t)\Vert_{\triangle x}^{2}\right)\, dt \\
\leq 2\int_0^T\left(\Vert B_{\triangle x}^{1/2}\theta_{\triangle x}(t)\Vert_{\triangle x}^{2}\, dt +(\triangle x)^\sigma\Vert \mathcal{V}_{\triangle x} ^{1/2}\theta_{\triangle x}(t)\Vert_{\triangle x}^{2}\right)\, dt \\
+2\int_0^T\left(\Vert B_{\triangle x}^{1/2}z_{\triangle x}(t)\Vert_{\triangle x}^{2}\, dt+(\triangle x)^\sigma\Vert \mathcal{V}_{\triangle x} ^{1/2}z_{\triangle x}(t)\Vert_{\triangle x}^{2}\right)\, dt,
\end{multline*}
and the lemma follows.
\end{proof}

\paragraph{Third step.} \textit{Nonlinear energy estimate.}

We define $w$ by \eqref{def_w}, as before, with $\beta>0$ to be chosen later.

\begin{lemma}\label{lemma3}
For every solution $u_{\triangle x}$ of \eqref{waveqSpaceLin}, we have
\begin{multline*} 
w(E_{\phi_{\triangle x}}(0)) \int_0^T \left(\Vert B_{\triangle x}^{1/2} u_{\triangle x}(t)\Vert_{\triangle x}^2+\Vert B_{\triangle x}^{1/2}F_{\triangle x}(u_{\triangle x}(t))\Vert_{\triangle x}^{2} + (\triangle x)^\sigma\Vert \mathcal{V}_{\triangle x}^{1/2} u_{\triangle x}(t)\Vert^2_{\triangle x}\right)\, dt  \\
\lesssim T \Vert B\Vert H^* (w(E_{\phi_{\triangle x}}(0))) \qquad\qquad\qquad\qquad\qquad\qquad\qquad\qquad\qquad\qquad\qquad\qquad\qquad \\
+ ( w(E_{\phi_{\triangle x}}(0))+1) \int_0^T\left(\langle B_{\triangle x}u_{\triangle x}(t),F_{\triangle x}(u_{\triangle x}(t))\rangle_{\triangle x}+(\triangle x)^\sigma\Vert \mathcal{V}_{\triangle x}^{1/2}u_{\triangle x}(t)\Vert_{\triangle x}^{2}\right)\, dt .
\end{multline*}
\end{lemma}


\begin{proof}
We set $f_{\triangle x}(t) = U^{-1} \tilde P_{\triangle x}u_{\triangle x}(t)$.
Using the definitions of the discretized operators and the isometry $U$, we have
\begin{multline*}
\Vert B_{\triangle x}^{1/2} u_{\triangle x}\Vert_{\triangle x}^2
= \langle B_{\triangle x}u_{\triangle x}, u_{\triangle x}\rangle_{\triangle x}
= \langle B \tilde P_{\triangle x}u_{\triangle x}, \tilde P_{\triangle x} u_{\triangle x}\rangle_X \\
= \langle U^{-1}B U U^{-1} \tilde P_{\triangle x}u_{\triangle x}, U^{-1}\tilde P_{\triangle x} u_{\triangle x}\rangle_{L^2(\Omega,\mu)}
= \langle b f_{\triangle x}, f_{\triangle x}\rangle_{L^2(\Omega,\mu)}
= \int_\Omega b f_{\triangle x}^2\, d\mu ,
\end{multline*}
and, using \eqref{def_rhotilde},
\begin{equation*}
\begin{split}
\Vert B_{\triangle x}^{1/2}F_{\triangle x}(u_{\triangle x})\Vert_{\triangle x}^{2}
& = \langle B_{\triangle x}F_{\triangle x}(u_{\triangle x}), F_{\triangle x}(u_{\triangle x})\rangle_{\triangle x}
= \langle B \tilde P_{\triangle x}F_{\triangle x}(u_{\triangle x}), \tilde P_{\triangle x} F_{\triangle x}(u_{\triangle x})\rangle_X  \\
&= \langle U^{-1} B U U^{-1} \tilde P_{\triangle x} P_{\triangle x} F(\tilde P_{\triangle x}u_{\triangle x}), U^{-1}\tilde P_{\triangle x}P_{\triangle x} F(\tilde P_{\triangle x}u_{\triangle x})\rangle_{L^2(\Omega,\mu)} \\
&= \langle b U^{-1} \tilde P_{\triangle x} P_{\triangle x} F(U U^{-1} \tilde P_{\triangle x}u_{\triangle x}), U^{-1}\tilde P_{\triangle x}P_{\triangle x} F(U U^{-1} \tilde P_{\triangle x}u_{\triangle x})\rangle_{L^2(\Omega,\mu)} \\
&= \langle b \tilde\rho_{\triangle x} (f_{\triangle x}), \tilde\rho_{\triangle x} (f_{\triangle x})\rangle_{L^2(\Omega,\mu)} \\
& = \int_\Omega b \left( \tilde\rho_{\triangle x} (f_{\triangle x})\right)^2 d\mu ,
\end{split}
\end{equation*}
and, with a similar computation,
\begin{equation*}
\langle B_{\triangle x}u_{\triangle x},F_{\triangle x}(u_{\triangle x})\rangle_{\triangle x}
= \int_\Omega b f_{\triangle x} \tilde\rho_{\triangle x} (f_{\triangle x}) \, d\mu .
\end{equation*}
Then, to prove the lemma, it suffices to prove that
\begin{multline*}
\int_0^T w(E_\phi(0)) \int_\Omega \left( b f_{\triangle x}^2 + b\left( \tilde\rho_{\triangle x}(f_{\triangle x}) \right)^2 \right) d\mu\, dt  \\
\lesssim T \int_\Omega b\, d\mu \  H^*(w(E_\phi(0)))  + \left( w(E_{\phi}(0)) + 1\right) \int_0^T \int_\Omega b f_{\triangle x} \tilde\rho_{\triangle x}(f_{\triangle x})\, d\mu\, dt    .
\end{multline*}
Indeed, the viscosity terms at the left-hand and right-hand sides of the desired inequality play no role in that lemma.
Then, the proof is completely similar to the one of Lemma \ref{cont_lem3}, using in particular the assumptions \eqref{assumptionfrhotildef} and \eqref{ineqgrhotilde}. We skip it.
\end{proof}

\paragraph{Fourth step.} \textit{End of the proof.}

\begin{lemma}\label{intermediateTheo}
We have
$$
E_{u_{\triangle x}}(T)\leq E_{u_{\triangle x}}(0)\left(1-\rho_TL^{-1}\left(\frac{E_{u_{\triangle x}}(0)}{\beta}\right)\right),
$$
for some positive constant $\rho_T<1$.
\end{lemma}

\begin{proof}
Recall that $E_{\phi_{\triangle x}}(0)=E_{u_{\triangle x}}(0)$ and that
\begin{equation}\label{or2126}
E_{u_{\triangle x}}(0)-E_{u_{\triangle x}}(T)=\int_0^T(\langle B_{\triangle x}u_{\triangle x}(t),F_{\triangle x}(u_{\triangle x}(t))\rangle_{\triangle x}+(\triangle x)^\sigma\Vert \mathcal{V}_{\triangle x}^{1/2}u_{\triangle x}(t)\Vert_{\triangle x}^{2})\, dt.
\end{equation}
Using successively the observability inequality \eqref{ineqObsDiscrete} and the estimates of Lemma \ref{lemma2}, of Lemma \ref{lemma1} and of Lemma \ref{lemma3}, we get
\begin{equation*}
\begin{split}
& 2C_Tw(E_{\phi_{\triangle x}}(0))E_{\phi_{\triangle x}}(0) \\
& \leq  w(E_{\phi_{\triangle x}}(0))\int_0^T\left(\Vert B_{\triangle x}^{1/2}\phi_{\triangle x}(t)\Vert_{\triangle x}^{2}+ (\triangle x)^\sigma \Vert \mathcal{V}_{\triangle x}^{1/2}\phi_{\triangle x}(t)\Vert_{\triangle x}^{2}\right)\, dt\\
 & \leq k_Tw(E_{\phi_{\triangle x}}(0))\int_0^T\left(\Vert B_{\triangle x}^{1/2}z_{\triangle x}(t)\Vert_{\triangle x}^{2}+(\triangle x)^\sigma\Vert \mathcal{V}_{\triangle x}^{1/2} z_{\triangle x}(t)\Vert_{\triangle x}^{2}\right)\, dt\\
 & \leq 2k_Tw(E_{\phi_{\triangle x}}(0))\int_0^T\left(\Vert B_{\triangle x}^{1/2}u_{\triangle x}(t)\Vert_{\triangle x}^{2}+\Vert B_{\triangle x}^{1/2}F_{\triangle x}(u_{\triangle x}(t))\Vert_{\triangle x}^2+(\triangle x)^\sigma\Vert \mathcal{V}_{\triangle x}^{1/2} u_{\triangle x}(t)\Vert_{\triangle x}^{2}\right)\, dt\\
 & \lesssim k_T T\Vert B\Vert H^* (w(E_{\phi_{\triangle x}}(0)))  \\
&\qquad + (w(E_{\phi_{\triangle x}}(0))+1)\int_0^T\left(\langle F_{\triangle x}(u_{\triangle x}(t)),B_{\triangle x}u_{\triangle x}(t)\rangle_{\triangle x}+(\triangle x)^\sigma\Vert \mathcal{V}_{\triangle x}^{1/2}u_{\triangle x}(t)\Vert_{\triangle x}^{2}\right)\, dt .
\end{split}
\end{equation*}
Using \eqref{or2126}, we infer that
$$
2C_Tw(E_{\phi_{\triangle x}}(0))E_{\phi_{\triangle x}}(0) 
\lesssim
k_T T\Vert B\Vert H^* (w(E_{\phi_{\triangle x}}(0))) + (w(E_{\phi_{\triangle x}}(0))+1) ( E_{u_{\triangle x}}(0)-E_{u_{\triangle x}}(T) ) .
$$
We conclude similarly as in the proof of Lemma \ref{cont_lem4}.
\end{proof}

The end of the proof of Theorem \ref{thm_space} follows the same lines as in the continuous case.

\subsection{Semi-discretization in time}\label{sec_semidiscretetime}
In this section, we analyze a time semi-discrete version of \eqref{main_eq}, with the objective of establishing uniform decay estimates. To this aim, following \cite{EZ1}, we add a suitable viscosity term in an implicit midpoint numerical scheme.

Given a solution $u$ of \eqref{main_eq}, for any $\triangle t>0$, we denote by $u^k$ the approximation of $u$ at time $t_k=k \triangle t$ with $k \in \N$. We consider the implicit midpoint time discretization of \eqref{main_eq} given by
\begin{equation}\label{time_dis_nl}
\left\{\begin{split}
& \frac{\widetilde{u}^{k+1} - u^k}{\triangle t} + A\left(\frac{u^k+ \widetilde{u}^{k+1}}{2}\right) + BF \left( \frac{u^k+ \widetilde{u}^{k+1}}{2}\right) = 0,\\
& \frac{  \widetilde{u}^{k+1} - u^{k+1} }{\triangle t} = \mathcal{V}_{\triangle t} u^{k+1} , \\
& u^0=u(0) .
\end{split}\right.
\end{equation}
The second equation in \eqref{time_dis_nl} is a viscosity term (see Remark \ref{rem1732} further for comments on this term and on the choice of the midpoint rule).
We only assume, throughout, that $\mathcal{V}_{\triangle t}:X\rightarrow X$ is a positive selfadjoint operator.

Written in an expansive way, \eqref{time_dis_nl} gives
\begin{equation*}
\frac{u^{k+1} - u^k}{\triangle t}  + A\left(\frac{u^k+ u^{k+1}  }{2}\right)  + BF \left( \frac{u^k+ (\mathrm{id}_X + \triangle t\, \mathcal{V}_{\triangle t}) u^{k+1} }{2}\right) 
+ \mathcal{V}_{\triangle t} u^{k+1}  + \frac{\triangle t}{2} A \mathcal{V}_{\triangle t} u^{k+1} = 0 .
\end{equation*}

We define the energy of a solution $(u^k)_{k\in\N}$ of \eqref{time_dis_nl} as the sequence $(E_{u^k})_{k\in\N}$ given by
\begin{equation}\label{defE_timediscrete}
E_{u^k} = \frac{1}{2}\Vert u^k\Vert _{X}^2 .
\end{equation}
As long as the solution is well defined, using the first equation in \eqref{time_dis_nl}, we have
\begin{equation*}\label{dissip1}
E_{\widetilde{u}^{k+1}} - E_{u^k} =  - \triangle t \left\langle B\left(\frac{u^k+ \widetilde{u}^{k+1}}{2}\right), F\left(\frac{u^k+ \widetilde{u}^{k+1}}{2}\right) \right\rangle_X ,
\end{equation*}
and using the second equation in \eqref{time_dis_nl}, we get that
\begin{equation*}\label{dissip2}
E_{\widetilde{u}^{k+1}} = E_{u^{k+1}} + \triangle t \, \Vert (\mathcal{V}_{\triangle t})^{1/2} u^{k+1}\Vert_X^2 + \frac{(\triangle t )^2}{2} \Vert \mathcal{V}_{\triangle t} u^{k+1}\Vert_X^2 .
\end{equation*}
We infer from these two relations that
\begin{multline}\label{dissip3}
E_{u^{k+1}} - E_{u^k} = - \triangle t \left\langle B\left(\frac{u^k+ \widetilde{u}^{k+1}}{2}\right), F\left(\frac{u^k+ \widetilde{u}^{k+1}}{2}\right) \right\rangle_X \\
- \triangle t \, \Vert (\mathcal{V}_{\triangle t})^{1/2} u^{k+1}\Vert_X^2 - \frac{(\triangle t )^2}{2} \Vert \mathcal{V}_{\triangle t} u^{k+1}\Vert_X^2 ,
\end{multline}
for every integer $k$.
Note that, thanks to \eqref{ineqBF}, we have $E_{u^{k+1}} - E_{u^k}\leq 0$, and therefore the energy defined by \eqref{defE_timediscrete} decays. We next perform an analysis similar to the one done in Section \ref{sec:Continuous}, but in the time semi-discrete setting, with the objective of deriving sharp decay estimates for  \eqref{waveqSpaceDiscrete}, which are uniform with respect to $\triangle t$ and $\triangle x$.

\subsubsection{Main result}
We keep all notations and assumptions done in Section \ref{sec:mainResConti}.

Hereafter, the notations $\lesssim$ and $\simeq$, already used in the previous sections, keep the same meaning as before, with the additional requirement that they also mean that the involved constants are uniform as well with respect to $\triangle t$.

The main result of this section is the following.

\begin{theorem}\label{thm_time}
In addition to the above assumptions, we assume that there exist $T>0$ and $C_T>0$ such that, setting $N=[T/\triangle t]$ (integer part), we have
\begin{multline}\label{unifobs}
C_T \Vert \phi^0\Vert_X^2 \leq \triangle t \sum_{k=0}^{N-1} \left\Vert  B^{1/2} \left( \frac{\phi^k+ \widetilde{\phi}^{k+1}}{2}\right)\right\Vert_X^2 \\
+ \triangle t \, \sum_{k=0}^{N-1} \Vert (\mathcal{V}_{\triangle t})^{1/2}\phi^{k+1}\Vert_X^2+
(\triangle t )^2 \sum_{k=0}^{N-1} \Vert \mathcal{V}_{\triangle t}\phi^{k+1}\Vert_X^2 ,
\end{multline}
for every solution of
\begin{equation}\label{time_dis_phi}
\left\{\begin{split}
& \frac{\widetilde{\phi}^{k+1} - \phi^k}{\triangle t} + A\left(\frac{\phi^k+ \widetilde{\phi}^{k+1}}{2}\right) = 0 , \\
& \frac{\widetilde{\phi}^{k+1}-\phi^{k+1}}{\triangle t} = \mathcal{V}_{\triangle t} \phi^{k+1} ,\end{split}\right.
\end{equation}
(uniform observability inequality with viscosity for the time semi-discretized linear conservative equation with viscosity).

Then, the solutions of \eqref{time_dis_nl} are well defined on $[0,+\infty)$ and, the energy of any solution satisfies
\begin{equation*}
E_{u^{k}} \lesssim  T \max(\gamma_1, E_{u^0}) \, L\left( \frac{1}{\psi^{-1}( \gamma_2 k \triangle t )} \right),
\end{equation*}
for every integer $k$, 
with $\gamma_2\simeq  C_T / T ( 1 + e^{2 T\Vert B\Vert} \max( 1, T\Vert B\Vert ) ) $
and $\gamma_1 \simeq \Vert B\Vert / \gamma_2$.
Moreover, under \eqref{condlimsup}, we have the simplified decay rate
\begin{equation*}
E_{u^{k}} \lesssim  T \max(\gamma_1, E_{u^0}) \, (H')^{-1}\left( \frac{\gamma_3}{k \triangle t } \right),
\end{equation*}
for every integer $k$, for some positive constant $\gamma_3\simeq 1$.
\end{theorem}

\begin{remark}\label{rem1732}
As in Remark \ref{rk:commentsdiscobsineq}, we insist on the crucial role of the viscosity.

In the absence of a viscosity term, the decay is not uniform in general: see for instance \cite[Theorem 5.1]{ZhangZhengZuazua} where a counterexample to uniform exponential stability (or equivalently, to the uniform observability estimate \eqref{unifobs} without viscosity) is given for a linear damped wave equation. As for space semi-discretizations, this is caused by spurious highfrequency modes that appear when discretizing in time, and that propagate with a vanishing velocity (as $\triangle t\rightarrow 0$). The role of the viscosity term is then to damp out these highfrequency spurious components. Note that other remedies to the lack of uniformity are proposed as well in \cite{EZ1,ZhangZhengZuazua}, for instance filtering the highfrequencies. As before, we focus here on the use of viscosity terms, more appropriate in our nonlinear context.

\medskip

The main assumption in Theorem \ref{thm_time} is the uniform observability inequality \eqref{unifobs}.

Certainly, the most general result, which can be directly used and adapted in our study, can be found in the remarkable article \cite{EZ1} (see also references therein, of which that paper is a far-reaching achievement), from which we infer the following typical example of a viscosity operator:
$$
\mathcal{V}_{\triangle t} = -(\triangle t)^2 A^2 = (\triangle t)^2 A^*A .
$$
For this choice, it is indeed proved in \cite[Lemma 2.4, and (1.17), (2.17) and (2.20)]{EZ1} that, if the observability inequality \eqref{observability_assumption} is valid for the continuous model, then the uniform observability inequality \eqref{unifobs} holds true for the time semi-discrete model \eqref{time_dis_phi}.
We could take as well the viscosity term $\mathcal{V}_{\triangle t} = - (\mathrm{id}_X- (\triangle t)^2 A^2)^{-1} (\triangle t)^2 A^2$, which yields as well \eqref{unifobs}, and which has the advantage of being bounded.

Note that, in contrast to space semi-discretizations (see Remark \ref{rk:commentsdiscobsineq}), here, for time semi-discretizations, the above choice of a viscosity systematically works in order to recover uniform properties.

\medskip

Apart from the viscosity term, note that the time discretization is an implicit midpoint rule. This choice is relevant for at least two reasons. The first is that an explicit time discretization would then immediately lead to a violation of the stability CFL condition, and thus the scheme is unstable. Actually, in Section \ref{sec_fullydiscrete}, we are going to consider full discretizations, and then, to be always in accordance with the CFL stability condition, it is better to choose an implicit scheme. This choice is also relevant with respect to the conservation of the energy, for the linear conservative equation \eqref{time_dis_phi}, at least, without the viscosity term.

Some variants of the midpoint rule and of the design of the viscosity term are possible (see \cite[Section 2.3]{EZ1}).
\end{remark}

\subsubsection{Proof of Theorem \ref{thm_time}}
We follow the lines of the proof of Theorem \ref{thm_continuous}.

As in the continuous setting and in the space semi-discrete setting, the global well-posedness of the solutions follows from usual a priori arguments (using energy estimates), and uniqueness follows from the assumption that $F$ is locally Lipschitz on bounded sets.

\paragraph{First step.} \textit{Comparison of the nonlinear equation \eqref{time_dis_nl} with the linear damped model.}

We first compare the nonlinear equation \eqref{time_dis_nl} with the linear damped equation with viscosity
\begin{equation}\label{time_dis_l}
\left\{\begin{split}
& \frac{\widetilde{z}^{k+1} - z^k}{\triangle t} + A\left(\frac{z^k+ \widetilde{z}^{k+1}}{2}\right) + B \left( \frac{z^k+ \widetilde{z}^{k+1}}{2}\right) = 0 ,\\
& \frac{ \widetilde{z}^{k+1} - z^{k+1} }{\triangle t} = \mathcal{V}_{\triangle t} z^{k+1} .
\end{split}\right.
\end{equation}

\begin{lemma}\label{t_dis_lem1}
For every solution $(u^k)_{k\in\N}$ of \eqref{time_dis_nl}, the solution of \eqref{time_dis_l} such that $z^0=u^0$ satisfies
\begin{multline*}
\sum_{k=0}^{N-1} \left\Vert B^{1/2} \left( \frac{z^k+ \widetilde{z}^{k+1}}{2}\right)\right\Vert_X^2 +  \sum_{k=0}^{N-1} \Vert (\mathcal{V}_{\triangle t})^{1/2} z^{k+1} \Vert_X^2 + \frac{\triangle t}{2}\sum_{k=0}^{N-1}  \Vert \mathcal{V}_{\triangle t} z^{k+1} \Vert_X^2 \\
\leq   
2  \sum_{k=0}^{N-1}\left\Vert B^{1/2} \left( \frac{u^k+ \widetilde{u}^{k+1}}{2}\right)\right\Vert_X^2
+ 2 \sum_{k=0}^{N-1} \left\Vert B^{1/2} F \left( \frac{u^k+ \widetilde{u}^{k+1}}{2}\right)\right\Vert_X^2 \\
+ 2 \sum_{k=0}^{N-1} \Vert (\mathcal{V}_{\triangle t})^{1/2} u^{k+1} \Vert_X^2 + \triangle t\, \sum_{k=0}^{N-1}  \Vert \mathcal{V}_{\triangle t} u^{k+1} \Vert_X^2  .
\end{multline*}
\end{lemma}

\begin{proof}
For every $k$, setting $\psi^k=u^k-z^k$ and $\widetilde{\psi}^k=\widetilde{u}^k - \widetilde{z}^k$, we have
\begin{equation}\label{dis_t_1}
\left\{\begin{split}
& \frac{\widetilde{\psi}^{k+1} - \psi^k}{\triangle t}=-A\left(\frac{\psi^k+ \widetilde{\psi}^{k+1}}{2}\right) -B\left(F \left( \frac{u^k+ \widetilde{u}^{k+1}}{2}\right)
- \frac{z^k + \widetilde{z}^k}{2}\right),\\
& \frac{\widetilde{\psi}^{k+1}-\psi^{k+1}}{\triangle t} = \mathcal{V}_{\triangle t} \psi^{k+1} .
\end{split}\right.
\end{equation}
Denoting $E_{\psi^k}=\frac{1}{2}\Vert \psi^k\Vert _{X}^2$, and taking the scalar product in $X$ in the first equation of \eqref{dis_t_1} with
$\frac{\widetilde{\psi}^{k+1} + \psi^k}{2}$, it follows that
\begin{equation*}
\begin{split}
& E_{\widetilde{\psi}^{k+1}} - E_{\psi^k} + \triangle t \left\langle B \left( \frac{z^k+ \widetilde{z}^{k+1}}{2}\right), \frac{z^k+ \widetilde{z}^{k+1}}{2}\right\rangle_X \\
&= - \triangle t \left\langle B F \left( \frac{u^k+ \widetilde{u}^{k+1}}{2}\right), \frac{u^k+ \widetilde{u}^{k+1}}{2}\right\rangle_X \\
& \qquad + \triangle t \left\langle B F \left( \frac{u^k+ \widetilde{u}^{k+1}}{2}\right),  \frac{z^k+ \widetilde{z}^{k+1}}{2}\right\rangle_X +
\triangle t \left\langle B \left( \frac{z^k+ \widetilde{z}^{k+1}}{2}\right),  \frac{u^k+ \widetilde{u}^{k+1}}{2}\right\rangle_X \\
&\leq 
\triangle t \left\langle B F \left( \frac{u^k+ \widetilde{u}^{k+1}}{2}\right),  \frac{z^k+ \widetilde{z}^{k+1}}{2}\right\rangle_X +
\triangle t \left\langle B \left( \frac{z^k+ \widetilde{z}^{k+1}}{2}\right),  \frac{u^k+ \widetilde{u}^{k+1}}{2}\right\rangle_X .
\end{split}
\end{equation*}
Using the second equation of \eqref{dis_t_1}, we infer that 
$$
E_{\widetilde{\psi}^{k+1}} - E_{\psi^{k+1}} = \triangle t\, \Vert (\mathcal{V}_{\triangle t})^{1/2} \psi^{k+1} \Vert_X^2 + \frac{(\triangle t)^2}{2} \Vert \mathcal{V}_{\triangle t} \psi^{k+1} \Vert_X^2 .
$$
Subtracting the latter equation to the previous one, we obtain that
\begin{multline*}
E_{\psi^{k+1}}- E_{\psi^k} + \triangle t \left\langle B \left( \frac{z^k+ \widetilde{z}^{k+1}}{2}\right), \frac{z^k+ \widetilde{z}^{k+1}}{2}\right\rangle_X +
\triangle t\, \Vert (\mathcal{V}_{\triangle t})^{1/2} \psi^{k+1} \Vert_X^2 + \frac{(\triangle t)^2}{2} \Vert \mathcal{V}_{\triangle t} \psi^{k+1} \Vert_X^2 \\
 \leq 
\triangle t \left\langle B F \left( \frac{u^k+ \widetilde{u}^{k+1}}{2}\right), \frac{z^k+ \widetilde{z}^{k+1}}{2} \right\rangle_X +
\triangle t \left\langle B \left( \frac{z^k+ \widetilde{z}^{k+1}}{2}\right), \frac{u^k+ \widetilde{u}^{k+1}}{2}\right\rangle_X .
\end{multline*}
Summing from $k=0$ to $k=N-1$, since $\psi^0=0$, we get that
\begin{multline*}
E_{\psi^{N}} + \triangle t \sum_{k=0}^{N-1} \left\langle B \left( \frac{z^k+ \widetilde{z}^{k+1}}{2}\right), \frac{z^k+ \widetilde{z}^{k+1}}{2}\right\rangle_X \\
+ \triangle t\, \sum_{k=0}^{N-1} \Vert (\mathcal{V}_{\triangle t})^{1/2} \psi^{k+1} \Vert_X^2 + \frac{(\triangle t)^2}{2}\sum_{k=0}^{N-1}  \Vert \mathcal{V}_{\triangle t} \psi^{k+1} \Vert_X^2 \\
\leq
\triangle t \sum_{k=0}^{N-1} \left\langle B F \left( \frac{u^k+ \widetilde{u}^{k+1}}{2}\right),  \frac{z^k+ \widetilde{z}^{k+1}}{2}\right\rangle_X +
\triangle t  \sum_{k=0}^{N-1}\left\langle B \left( \frac{z^k+ \widetilde{z}^{k+1}}{2}\right),  \frac{u^k+ \widetilde{u}^{k+1}}{2}\right\rangle_X.
\end{multline*}
Thanks to the Young inequality, and since $E_{\psi^{N}}\geq 0$, we infer that
\begin{multline*}
 \frac{\triangle t }{2}\sum_{k=0}^{N-1} \left\Vert B^{1/2} \left( \frac{z^k+ \widetilde{z}^{k+1}}{2}\right)\right\Vert_X^2 + \triangle t\, \sum_{k=0}^{N-1} \Vert (\mathcal{V}_{\triangle t})^{1/2} \psi^{k+1} \Vert_X^2 + \frac{(\triangle t)^2}{2}\sum_{k=0}^{N-1}  \Vert \mathcal{V}_{\triangle t} \psi^{k+1} \Vert_X^2 \\
\leq
\triangle t \sum_{k=0}^{N-1} \left\Vert B^{1/2} F \left( \frac{u^k+ \widetilde{u}^{k+1}}{2}\right)\right\Vert_X^2 +
\triangle t  \sum_{k=0}^{N-1}\left\Vert B^{1/2} \left( \frac{u^k+ \widetilde{u}^{k+1}}{2}\right)\right\Vert_X^2.
\end{multline*}
The lemma follows, using that $\Vert (\mathcal{V}_{\triangle t})^{1/2}\psi^{k+1}\Vert_X^2 \geq \frac{1}{2}\Vert (\mathcal{V}_{\triangle t})^{1/2} z^{k+1}\Vert_X^2 - \Vert  (\mathcal{V}_{\triangle t})^{1/2} u^{k+1}\Vert_X^2$ and that $\Vert \mathcal{V}_{\triangle t}\psi^{k+1}\Vert_X^2 \geq \frac{1}{2}\Vert \mathcal{V}_{\triangle t} z^{k+1}\Vert_X^2 - \Vert  \mathcal{V}_{\triangle t} u^{k+1}\Vert_X^2 $.
\end{proof}

\paragraph{Second step.} \textit{Comparison of the linear damped equation \eqref{time_dis_l} with the time semi-discretized conservative linear equation with viscosity \eqref{time_dis_phi}.}

\begin{lemma}\label{lemmadis_time}
For every solution $(z^k)_{k\in\N}$ of \eqref{time_dis_l}, the solution $(\phi^k)_{k\in\N}$ of \eqref{time_dis_phi} such that $\phi^{0}=z^0$ satisfies
\begin{multline*}
\frac{1}{2}\sum_{k=0}^{N-1} \left\Vert  B^{1/2} \left( \frac{\phi^k+ \widetilde{\phi}^{k+1}}{2}\right)\right\Vert_X^2 
+ \sum_{k=0}^{N-1} \Vert (\mathcal{V}_{\triangle t})^{1/2} \phi^{k+1}\Vert_X^2+
\frac{\triangle t}{2}\sum_{k=0}^{N-1} \Vert \mathcal{V}_{\triangle t} \phi^{k+1}\Vert_X^2 \\
\lesssim 
k_T\left(    \sum_{k=0}^{N-1} \left\Vert  B^{1/2} \left( \frac{z^k+ \widetilde{z}^{k+1}}{2}\right)\right\Vert_X^2  + \sum_{k=0}^{N-1} \Vert (\mathcal{V}_{\triangle t})^{1/2} z^{k+1}\Vert_X^2+
\frac{\triangle t}{2} \sum_{k=0}^{N-1} \Vert \mathcal{V}_{\triangle t} z^{k+1}\Vert_X^2 \right) ,
\end{multline*}
with 
$k_T=\max(1+(4T^2+1)^2||B^{1/2}||^4,2)$.
\end{lemma}

\begin{proof}
Setting $\theta^i=\phi^i-z^i$ and $\widetilde{\theta}^i=\widetilde{\phi}^i-\widetilde{z}^i$, we have
\begin{equation}\label{time_dis_theta}
\left\{\begin{split}
& \frac{\widetilde{\theta}^{i+1} - \theta^i}{\triangle t} + A\left(\frac{\theta^i+ \widetilde{\theta}^{i+1}}{2}\right) - B \left( \frac{z^i+ \widetilde{z}^{i+1}}{2}\right) = 0 ,  \\
& \frac{ \widetilde{\theta}^{i+1}- \theta^{i+1}}{\triangle t} = \mathcal{V}_{\triangle t} \theta^{i+1} .
\end{split}\right.
\end{equation}
Taking the scalar product in $X$ in the first equation of \eqref{time_dis_theta} with $\frac{1}{2}(\widetilde{\theta}^{i+1} + \theta^i)$, we get
\begin{equation}\label{int1}
E_{\widetilde{\theta}^{i+1}} - E_{{\theta}^{i}}=\triangle t \left\langle B \left( \frac{z^i+ \widetilde{z}^{i+1}}{2}\right),  \frac{\theta^i+ \widetilde{\theta}^{i+1}}{2}\right\rangle_X .
\end{equation}
Now, using the second equation in \eqref{time_dis_theta}, we obtain that
\begin{equation}\label{int2}
E_{\widetilde{\theta}^{i+1}} - E_{{\theta}^{i+1}}=\triangle t \, \Vert (\mathcal{V}_{\triangle t})^{1/2} \theta^{i+1}\Vert_X^2 + \frac{(\triangle t )^2}{2} \Vert \mathcal{V}_{\triangle t} \theta^{i+1}\Vert_X^2 .
\end{equation}
Subtracting \eqref{int2} to \eqref{int1}, and then summing from $i=0$ to $i=k-1$, with $k \leq N$, using that $\theta^0=0$, we obtain
\begin{multline}\label{orl1235}
E_{\theta^{k}} + \triangle t\, \sum_{i=0}^{k-1} \Vert (\mathcal{V}_{\triangle t})^{1/2}\theta^{i+1}\Vert_X^2 + \frac{(\triangle t )^2}{2}  \sum_{i=0}^{k-1} \Vert \mathcal{V}_{\triangle t}\theta^{i+1}\Vert_X^2 \\
= \triangle t \sum_{i=0}^{k-1} \left\langle B \left( \frac{z^i+ \widetilde{z}^{i+1}}{2}\right),  \frac{\theta^i+ \widetilde{\theta}^{i+1}}{2}\right\rangle_X .
\end{multline}
In passing, note that \eqref{orl1235} implies that
\begin{equation}\label{orl1318}
\sum_{k=0}^{N-1} \Vert (\mathcal{V}_{\triangle t})^{1/2}\theta^{k+1}\Vert_X^2 + \frac{\triangle t}{2}  \sum_{k=0}^{N-1} \Vert \mathcal{V}_{\triangle t}\theta^{k+1}\Vert_X^2
\leq \sum_{k=0}^{N-1} \left\langle B \left( \frac{z^k+ \widetilde{z}^{k+1}}{2}\right),  \frac{\theta^k+ \widetilde{\theta}^{k+1}}{2}\right\rangle_X .
\end{equation}
Now, summing \eqref{orl1235} from $k=0$ to $k=N-1$, using that $N\triangle t\leq T$, we get
\begin{multline*}
\sum_{k=0}^{N-1} E_{\theta^{k}} + \triangle t\, \sum_{k=0}^{N-1}  (N-1-k) \Vert (\mathcal{V}_{\triangle t})^{1/2}\theta^{k+1}\Vert_X^2 + \frac{(\triangle t )^2}{2} \sum_{k=0}^{N-1} (N-1-k) \Vert \mathcal{V}_{\triangle t}\theta^{k+1}\Vert_X^2 \\
= \triangle t \, \sum_{k=0}^{N-1} (N-1-k) \left\langle B \left( \frac{z^k+ \widetilde{z}^{k+1}}{2}\right),  \frac{\theta^k+ \widetilde{\theta}^{k+1}}{2}\right\rangle_X 
\leq T \sum_{k=0}^{N-1} \left\langle B \left( \frac{z^k+ \widetilde{z}^{k+1}}{2}\right),  \frac{\theta^k+ \widetilde{\theta}^{k+1}}{2}\right\rangle_X ,
\end{multline*}
from which it follows in particular that
\begin{equation}\label{orl1339}
\sum_{k=0}^{N-1} E_{\theta^{k}} \leq T \sum_{k=0}^{N-1} \left\langle B \left( \frac{z^k+ \widetilde{z}^{k+1}}{2}\right),  \frac{\theta^k+ \widetilde{\theta}^{k+1}}{2}\right\rangle_X .
\end{equation}
Besides, thanks to \eqref{int2}, since $\theta^0=0$, we have
\begin{equation*}
\sum_{k=0}^{N-2} E_{\widetilde{\theta}^{k+1}} =  \sum_{k=0}^{N-1} E_{{\theta}^k}+ \triangle t \, \sum_{k=0}^{N-2} \Vert (\mathcal{V}_{\triangle t})^{1/2} \theta^{k+1}\Vert_X^2 + \frac{(\triangle t )^2}{2} \sum_{k=0}^{N-2} \Vert \mathcal{V}_{\triangle t}\theta^{k+1}\Vert_X^2 ,
\end{equation*}
and, using \eqref{orl1318}, we infer that
\begin{equation}\label{orl1335}
\sum_{k=0}^{N-2} E_{\widetilde{\theta}^{k+1}}
\leq  \sum_{k=0}^{N-1} E_{{\theta}^k}+ \triangle t \, \sum_{k=0}^{N-1} \left\langle B \left( \frac{z^k+ \widetilde{z}^{k+1}}{2}\right),  \frac{\theta^k+ \widetilde{\theta}^{k+1}}{2}\right\rangle_X ,
\end{equation}
Using \eqref{int1} for $i=N-1$, we have
\begin{equation}\label{intX4}
E_{\widetilde{\theta}^N} = E_{\theta^{N-1}} + \triangle t \left\langle B \left( \frac{z^{N-1}+ \widetilde{z}^{N}}{2}\right),  \frac{\theta^{N-1}+ \widetilde{\theta}^{N}}{2}\right\rangle_X .
\end{equation}
Summing \eqref{orl1335} and \eqref{intX4}, we obtain
\begin{equation}\label{orl1338}
\sum_{k=0}^{N-1} E_{\widetilde{\theta}^{k+1}}
\leq  2\sum_{k=0}^{N-1} E_{{\theta}^k}+ 2\triangle t \, \sum_{k=0}^{N-1} \left\langle B \left( \frac{z^k+ \widetilde{z}^{k+1}}{2}\right),  \frac{\theta^k+ \widetilde{\theta}^{k+1}}{2}\right\rangle_X ,
\end{equation}
Finally, summing three times \eqref{orl1339} and \eqref{orl1338}, we get
\begin{equation}\label{orl1343}
\sum_{k=0}^{N-1} ( E_{\theta^{k}} + E_{\widetilde{\theta}^{k+1}} )
\leq  4T \sum_{k=0}^{N-1} \left\langle B \left( \frac{z^k+ \widetilde{z}^{k+1}}{2}\right),  \frac{\theta^k+ \widetilde{\theta}^{k+1}}{2}\right\rangle_X
\end{equation}
Noting that
$
\left\Vert \frac{\theta^k+ \widetilde{\theta}^{k+1}}{2}\right\Vert_X^2 
\leq \frac{1}{2} \left( \Vert \theta^k\Vert_X^2 + \Vert \widetilde{\theta}^{k+1} \Vert_X^2  \right)  =  E_{\theta^{k}} + E_{\widetilde{\theta}^{k+1}} ,
$
we infer from \eqref{orl1318} and \eqref{orl1338} that
\begin{multline}\label{orl1343}
\sum_{k=0}^{N-1} \left\Vert \frac{\theta^k+ \widetilde{\theta}^{k+1}}{2}\right\Vert_X^2  + \sum_{k=0}^{N-1} \Vert (\mathcal{V}_{\triangle t})^{1/2}\theta^{k+1}\Vert_X^2 + \frac{\triangle t}{2}  \sum_{k=0}^{N-1} \Vert \mathcal{V}_{\triangle t}\theta^{k+1}\Vert_X^2 \\
\leq  (4T+1) \sum_{k=0}^{N-1} \left\langle B \left( \frac{z^k+ \widetilde{z}^{k+1}}{2}\right),  \frac{\theta^k+ \widetilde{\theta}^{k+1}}{2}\right\rangle_X .
\end{multline}
By the Young inequality, we have
\begin{multline*}
(4T+1) \sum_{k=0}^{N-1}  \left\langle B \left( \frac{z^k+ \widetilde{z}^{k+1}}{2}\right),  \frac{\theta^k+ \widetilde{\theta}^{k+1}}{2}\right\rangle_X \\
\leq \frac{1}{2}(4T+1)^2 \Vert B^{1/2}\Vert^2 \sum_{k=0}^{N-1}  \left\Vert B^{1/2} \left( \frac{z^k+ \widetilde{z}^{k+1}}{2}\right) \right\Vert_X^2 +  \frac{1}{2} \sum_{k=0}^{N-1}  \left\Vert \frac{\theta^k+ \widetilde{\theta}^{k+1}}{2}\right\Vert_X^2  ,
\end{multline*}
and therefore we get from \eqref{orl1343} that
\begin{multline}\label{a1641}
\frac{1}{2} \sum_{k=0}^{N-1} \left\Vert \frac{\theta^k+ \widetilde{\theta}^{k+1}}{2}\right\Vert_X^2 + \sum_{k=0}^{N-1}  \Vert (\mathcal{V}_{\triangle t})^{1/2}\theta^{k+1}\Vert_X^2 + \frac{\triangle t}{2} \sum_{k=0}^{N-1} \Vert \mathcal{V}_{\triangle t}\theta^{k+1}\Vert_X^2 \\
\leq  
\frac{1}{2}(4T+1)^2 \Vert B^{1/2}\Vert^2 \sum_{k=0}^{N-1}  \left\Vert B^{1/2} \left( \frac{z^k+ \widetilde{z}^{k+1}}{2}\right) \right\Vert_X^2 .
\end{multline}
Now, since $\phi^k = \theta^k+z^k$ and $\widetilde{\phi}^k = \widetilde{\theta}^k+\widetilde{z}^k$, using the inequality $(a+b)^2\leq 2(a^2+b^2)$, we have
\begin{multline*}
\frac{1}{2} \sum_{k=0}^{N-1} \left\Vert B^{1/2} \left( \frac{\phi^k+ \widetilde{\phi}^{k+1}}{2} \right) \right\Vert_X^2 + \sum_{k=0}^{N-1}  \Vert (\mathcal{V}_{\triangle t})^{1/2}\phi^{k+1}\Vert_X^2 + \frac{\triangle t}{2} \sum_{k=0}^{N-1} \Vert \mathcal{V}_{\triangle t}\phi^{k+1}\Vert_X^2 \\
\leq
\Vert B^{1/2}\Vert^2  \sum_{k=0}^{N-1} \left\Vert \frac{\theta^k+ \widetilde{\theta}^{k+1}}{2}  \right\Vert_X^2 + 2 \sum_{k=0}^{N-1}  \Vert (\mathcal{V}_{\triangle t})^{1/2}\theta^{k+1}\Vert_X^2 + \triangle t \, \sum_{k=0}^{N-1} \Vert \mathcal{V}_{\triangle t}\theta^{k+1}\Vert_X^2 \\
+ \sum_{k=0}^{N-1} \left\Vert B^{1/2} \left( \frac{z^k+ \widetilde{z}^{k+1}}{2} \right) \right\Vert_X^2 + 2 \sum_{k=0}^{N-1}  \Vert (\mathcal{V}_{\triangle t})^{1/2}z^{k+1}\Vert_X^2 + \triangle t \, \sum_{k=0}^{N-1} \Vert \mathcal{V}_{\triangle t}z^{k+1}\Vert_X^2  ,
\end{multline*}
and, using \eqref{a1641}, the lemma follows.
\end{proof}

\paragraph{Third step.} \textit{Nonlinear energy estimate.}

We define $w$ by \eqref{def_w}, as before, with $\beta>0$ to be chosen later.

\begin{lemma}\label{dis_time_lem3}
For every solution $(u^k)_{k\in\N} $ of \eqref{time_dis_nl}, we have
\begin{multline*}
\triangle t \sum_{k=0}^{N-1} w(E_{\phi^0})\left( \left\Vert B^{1/2} \left(\frac{u^k+ \widetilde{u}^{k+1}}{2}\right)\right\Vert_X^2 + 
 \left\Vert B^{1/2}F \left(\frac{u^k+ \widetilde{u}^{k+1}}{2}\right)\right\Vert_X^2 \right) 
\\
\lesssim \Vert B\Vert  N \triangle t \,H^*(w(E_{\phi^0})) +
(w(E_{\phi^0})+1) \triangle t  \sum_{k=0}^{N-1} \left\langle B\left(\frac{u^k+ \widetilde{u}^{k+1}}{2}\right), F\left(\frac{u^k+ \widetilde{u}^{k+1}}{2}\right) \right\rangle_X .
\end{multline*}
\end{lemma}

\begin{proof}
For every $k\in\N$, we set $f^k=U^{-1}u^k$ and $\widetilde{f}^k=U^{-1}\widetilde{u}^k$. By definition of $\rho$, we have $\rho(f^k)=U^{-1}F(Uf^k)$ for every $k \in \N$. We set $\varepsilon_{0}=\min(1, g(s_0))$ and we define, for every $k \in \N$, the set $\Omega_1^k=\{ x \in \Omega \ | \ |\frac{f^k+ \widetilde{f}^{k+1}}{2}(x)| \leq \varepsilon_{0}\}$. Using \eqref{ineqg}, we have $\frac{1}{c_2^2} \rho\left(\frac{f^k+ \widetilde{f}^{k+1}}{2}\right)^2 \in [0,s_0^2]$ and
$$
\frac{1}{\int_{\Omega_1^k} b\, d\mu} \int_{\Omega_1^k} \frac{1}{c_2^2} \rho\left(\frac{f^k+ \widetilde{f}^{k+1}}{2}\right)^2 b \,d\mu \in  [0,s_0^2] .
$$
Since $H$ is convex on $[0,s_0^2]$, applying the Jensen inequality, we get
\begin{multline*}
H\left(\frac{1}{\int_{\Omega_1^k} b\, d\mu} \int_{\Omega_1^k} \frac{1}{c_2^2} \rho\left(\frac{f^k+ \widetilde{f}^{k+1}}{2}\right)^2 b \,d\mu \right) \\
\leq \frac{1}{\int_{\Omega_1^k} b\, d\mu} \int_{\Omega_1^k} \frac{1}{c_2}\left| \rho\left(\frac{f^k+ \widetilde{f}^{k+1}}{2}\right) \right| g\left( \frac{1}{c_2} \left\vert \rho\left(\frac{f^k+ \widetilde{f}^{k+1}}{2}\right)\right\vert \right)b\, d\mu .
\end{multline*}
Using \eqref{ineqg} and the sign condition \eqref{assumptionfrhof}, we infer that
\begin{multline*}
H\left(\frac{1}{\int_{\Omega_1^k} b\, d\mu} \int_{\Omega_1^k} \frac{1}{c_2^2} \rho^2\left(\frac{f^k+ \widetilde{f}^{k+1}}{2}\right) b \,d\mu
\right) \leq \frac{1}{c_2}\frac{1}{\int_{\Omega_1^k} b\, d\mu} \int_{\Omega_1^k} b \frac{f^k+ \widetilde{f}^{k+1}}{2}  \rho\left(\frac{f^k+ \widetilde{f}^{k+1}}{2}\right) \, d\mu  .
\end{multline*}
Since $H$ is increasing, we deduce that
\begin{multline*}
\int_{\Omega_1^k}  \rho\left(\frac{f^k+ \widetilde{f}^{k+1}}{2}\right)^2 b \,d\mu
\leq c_2^2 \int_{\Omega_1^k} b\, d\mu \ H^{-1}\left( \frac{1}{c_2}\frac{1}{\int_{\Omega_1^k} b\, d\mu} \int_{\Omega_1^k} b \frac{f^k+ \widetilde{f}^{k+1}}{2} \rho\left(\frac{f^k+ \widetilde{f}^{k+1}}{2}\right) d\mu\right) ,
\end{multline*}
and therefore,
\begin{multline*}
\triangle t \sum_{k=0}^{N-1}  w(E_{\phi^0})  \int_{\Omega_1^k}  \rho\left(\frac{f^k+ \widetilde{f}^{k+1}}{2}\right)^2 b \,d\mu \\
\leq  \triangle t \sum_{k=0}^{N-1} c_2^2  \int_{\Omega_1^k} b\, d\mu \ w(E_{\phi^0})
H^{-1}\left( \frac{1}{c_2}\frac{1}{\int_{\Omega_1^k} b\, d\mu} \int_{\Omega_1^k} b \frac{f^k+ \widetilde{f}^{k+1}}{2} \rho\left(\frac{f^k+ \widetilde{f}^{k+1}}{2}\right) \, d\mu\right) .
\end{multline*}
Hence, according to the Young inequality $AB\leq H(A)+H^*(B)$, we get that
\begin{multline}\label{fk_nl}
\triangle t \sum_{k=0}^{N-1}  w(E_{\phi^0})  \int_{\Omega_1^k}  \rho\left(\frac{f^k+ \widetilde{f}^{k+1}}{2}\right)^2 b \, d\mu \leq
c_2 \triangle t  \sum_{k=0}^{N-1} \int_{\Omega} b \frac{f^k+ \widetilde{f}^{k+1}}{2} \rho\left(\frac{f^k+ \widetilde{f}^{k+1}}{2}\right) \, d\mu \\
+  N \triangle t \, c_2^2 \int_{\Omega} b\, d\mu \  H^*(w(E_{\phi^0})) .
\end{multline}
Besides, in $\Omega \setminus \Omega_1^k$, using \eqref{ineqg} we have 
$\left\vert \rho\left(\frac{f^k+ \widetilde{f}^{k+1}}{2}\right)\right\vert \lesssim  \left|\frac{f^k+ \widetilde{f}^{k+1}}{2}\right|$, and using \eqref{assumptionfrhof}, it follows that
$$
\int_{\Omega \setminus \Omega_1^k}  \rho\left(\frac{f^k+ \widetilde{f}^{k+1}}{2}\right)^2 b \,d\mu 
\lesssim \int_{\Omega} b \frac{f^k+ \widetilde{f}^{k+1}}{2}  \rho\left(\frac{f^k+ \widetilde{f}^{k+1}}{2}\right) \, d\mu .
$$
Using this inequality and \eqref{fk_nl}, we obtain
\begin{multline*}
\triangle t \sum_{k=0}^{N-1} w(E_{\phi^0})  \int_{\Omega}  \rho\left(\frac{f^k+ \widetilde{f}^{k+1}}{2}\right)^2 b \,d\mu
\lesssim N \triangle t \int_{\Omega} b\, d\mu \ H^*(w(E_{\phi^0})) \\
+ ( w(E_{\phi^0}) + 1)  \triangle t  \sum_{k=0}^{N-1}\int_{\Omega} b \frac{f^k+ \widetilde{f}^{k+1}}{2} \rho\left(\frac{f^k+ \widetilde{f}^{k+1}}{2}\right) \, d\mu .
\end{multline*}
Now, defining $\varepsilon_1$ as in Lemma \ref{cont_lem3}, we set $\Omega_2^k=\{ x \in \Omega \ | \ |\frac{f^k+ \widetilde{f}^{k+1}}{2}(x)| \leq \varepsilon_{1}\}$. Using $\eqref{ineqg}$, we have $\left|\frac{f^k+ \widetilde{f}^{k+1}}{2}\right|^2 \in [0,s_0^2]$ and
$$
\frac{1}{\int_{\Omega_2^k} b\, d\mu} \int_{\Omega_2^k} \left|\frac{f^k+ \widetilde{f}^{k+1}}{2}\right|^2 b \,d\mu \in  [0,s_0^2] .
$$
Since $H$ is convex on $[0,s_0^2]$, by the Jensen inequality, we get
\begin{equation*}
\begin{split}
H\left(\frac{1}{\int_{\Omega_2^k} b\, d\mu} \int_{\Omega_1^k} \left|\frac{f^k+ \widetilde{f}^{k+1}}{2}\right|^2 b \,d\mu \right) 
&\leq \frac{1}{\int_{\Omega_2^k} b\, d\mu} \int_{\Omega_2^k} \left|\frac{f^k+ \widetilde{f}^{k+1}}{2} \right| g\left( \left|\frac{f^k+ \widetilde{f}^{k+1}}{2}\right| \right)b\, d\mu  \\
&\leq \frac{1}{c_1} \frac{1}{\int_{\Omega_2^k} b\, d\mu} \int_{\Omega} b \frac{f^k+ \widetilde{f}^{k+1}}{2} \rho \left( \frac{f^k+ \widetilde{f}^{k+1}}{2}\right) \, d\mu .
\end{split}
\end{equation*}
Since $H$ is increasing on $[0,s_0^2]$, we deduce that
\begin{multline*}
\triangle t\, \sum_{k=0}^{N-1} w(E_{\phi^0})  \int_{\Omega_2^k}  \left|\frac{f^k+ \widetilde{f}^{k+1}}{2}\right|^2 b \,d\mu \\
\leq \triangle t \sum_{k=0}^{N-1} w(E_{\phi^0})  \int_{\Omega_2^k} b\, d\mu \ H^{-1}\left( \frac{1}{c_1}\frac{1}{\int_{\Omega_2^k} b\, d\mu} \int_{\Omega} b \frac{f^k+ \widetilde{f}^{k+1}}{2} \rho\left(\frac{f^k+ \widetilde{f}^{k+1}}{2}\right) \, d\mu\right) .
\end{multline*}
It then follows from the Young inequality $AB\leq H(A)+H^*(B)$ that
\begin{multline*}
\triangle t \, \sum_{k=0}^{N-1} w(E_{\phi^0})   \int_{\Omega_2^k} \left| \frac{f^k+ \widetilde{f}^{k+1}}{2}\right|^2 b \,d\mu \leq
N \triangle t \, \int_{\Omega} b\, d\mu \ H^*(w(E_{\phi^0})) \\
+ \frac{1}{c_1} \triangle t\, \sum_{k=0}^{N-1}  \int_{\Omega} b \frac{f^k+ \widetilde{f}^{k+1}}{2} \rho\left(\frac{f^k+ \widetilde{f}^{k+1}}{2}\right) \, d\mu .
\end{multline*}
The estimate in $\Omega \setminus \Omega_2^k$ is obtained similarly.
The lemma is proved.
\end{proof}

\paragraph{Fourth step.} \textit{End of the proof.}

\begin{lemma}\label{intermediateTheo3}
We have
\begin{equation}\label{estENast}
E_{u^{N}} \leq E_{u^0}\left(1-\rho_TL^{-1}\left(\frac{E_{u^0}}{\beta}\right)\right),
\end{equation}
for some sufficiently small positive constant $\rho_T$.
\end{lemma}

\begin{proof}
Summing in $k$ the energy dissipation relations \eqref{dissip3}, we get
\begin{multline}\label{dissip4}
E_{u^{N}} - E_{u^0}
= - \triangle t \sum_{k=0}^{N-1}  \left\langle B\left(\frac{u^k+ \widetilde{u}^{k+1}}{2}\right), F\left(\frac{u^k+ \widetilde{u}^{k+1}}{2}\right) \right\rangle_X \\
- \triangle t \, \sum_{k=0}^{N-1} \Vert (\mathcal{V}_{\triangle t})^{1/2} u^{k+1}\Vert_X^2 - \frac{(\triangle t )^2}{2} \sum_{k=0}^{N-1}\Vert \mathcal{V}_{\triangle t} u^{k+1}\Vert_X^2 .
\end{multline}
Using successively the observability inequality \eqref{unifobs} and the estimates obtained in Lemmas \ref{lemmadis_time}, \ref{t_dis_lem1}, and \ref{dis_time_lem3}, and since $N \triangle t \leq T$, we get
\begin{multline*}
C_T w(E_{\phi^0}) \Vert \phi^0\Vert_X^2 \\
\lesssim
k_T \Vert B\Vert T H^*(w(E_{\phi^0})) +
k_T (w(E_{\phi^0})+1) \triangle t  \sum_{k=0}^{N-1} \left\langle B\left(\frac{u^k+ \widetilde{u}^{k+1}}{2}\right), F\left(\frac{u^k+ \widetilde{u}^{k+1}}{2}\right) \right\rangle_X   \\
+ k_T \triangle t \, w(E_{\phi^0}) \left(   \sum_{k=0}^{N-1} \Vert (\mathcal{V}_{\triangle t})^{1/2} u^{k+1} \Vert_X^2 + \triangle t\, \sum_{k=0}^{N-1}  \Vert \mathcal{V}_{\triangle t} u^{k+1} \Vert_X^2  \right) .
\end{multline*}
Recalling that $E_{\phi^0}=E_{u^0}$, using \eqref{dissip4}, we infer that
\begin{equation*}
C_T w(E_{u^0}) E_{u^0}  \lesssim k_T  \Vert B\Vert T \, H^*(w(E_{u^0})) + k_T (w(E_{u^0})+1) \left( E_{u^0} - E_{u^{N}} \right).
\end{equation*}
Then, reasoning as in the proof of Lemma \ref{cont_lem4}, we get the conclusion (we skip the details).
\end{proof}

Let now $p \in \N$ be arbitrary and $k\in \{0, \ldots, N\}$. The sequence
$(v^k)_{k \in \{0,\ldots, N-1\}}$ defined by $v^k=u^{k+pN}$ for $k \in \{0,\ldots, N-1\}$, satisfies \eqref{time_dis_nl} and thus \eqref{estENast}. Noting then that
$E_{v^{N}}=E_{u^{(p+1)N}}$ and $E_{v^0}=E_{u^{pN}}$, we deduce that $E_{k+1} - E_k + \rho_T M(E_k) \leq 0$, where $M(x) = xL^{-1}(x)$ for every $x\in[0,s_0^2]$ and $E_k=E_{u^{k N}}/\beta$.

As at the end of Section \ref{sec:proofthm_continuous}, setting $K_r(\tau)=\int_\tau^r \frac{dy}{M(y)}$, we have
$$
M(E_p) \leq \frac{1}{\rho_T} \min_{\ell\in\{0,\ldots,p\}} \left(  \frac{K_r^{-1}(\rho_T(p-\ell)) }{\ell+1} \right) .
$$
We set $t=pT$. For any $\theta \in (0,t]$, we set $l=\left[\frac{\theta}{T}\right] \in \{0,\ldots,p\}$. We have
$$
E_p \leq M^{-1} \left( \frac{T}{\rho_T}\inf_{0<\theta\leq t} \left( \frac{1}{\theta}K_r^{-1}\left( \rho_T\frac{t-\theta}{T} \right) \right)  \right).
$$
As in the proof of \cite[Theorem 2.1]{alabau_AMO2005}, we deduce that
\begin{equation}\label{estEupNg}
E_{u^{pN}} \leq  \beta L\left( \frac{1}{\psi^{-1}\left(\rho_T p \right)} \right),
\end{equation}
for every $p\geq 1/(\rho_TH'(s_0^2))$, with $\psi$ defined by \eqref{defpsir}. 

Moreover, under \eqref{condlimsup}, we get that $E_{u^{pN}} \leq  \beta (H')^{-1}\left( \frac{\beta_3}{\rho_T p} \right)$, for every $p$ sufficiently large, for a certain $\gamma_3>0$ not depending on $p$, $\triangle t$, $E_0$. 
Let now $k >T/\triangle t$ be a given integer. We set $p=k/N$. Since $p N \leq k$ and thanks to the dissipation property \eqref{dissip3}, we have $E_{u^k} \leq E_{u^{pN}}$. Besides, we have $\rho_T p > \rho_T(k/N-1)$. Since $\psi$ and $L$ are nondecreasing and thanks to \eqref{estEupNg}, it follows that
\begin{equation*}
E_{u^k} \leq \beta L\left( \frac{1}{\psi^{-1}\left(\frac{\rho_T}{T}\left(k \triangle t - T\right)\right)} \right),
\end{equation*}
for every $k\geq \frac{1}{\rho_TH'(s_0^2)}\frac{T}{\triangle t}$. Moreover, under \eqref{condlimsup}, we have
$E_{u^{k}} \leq  \beta (H')^{-1}\left( \frac{\beta_3}{\frac{\rho_T}{T}\left(k \triangle t - T\right)} \right)$, for $k$ sufficiently large. 
Theorem \ref{thm_time} is proved.

\subsection{Full discretization}\label{sec_fullydiscrete}
Following \cite{EZ1} (see also references therein), results for full discretization schemes may be obtained from the previous time discretization and space discretization results, as follows: it suffices to notice that the results for time semi-discrete approximation schemes are actually valid for a class of abstract systems depending on a parameter, uniformly with respect to this parameter that is typically the space mesh parameter $\triangle x$. Then, using the results obtained for space semi-discretizations, we infer the desired uniform properties for fully discrete schemes.

More precisely, the class of abstract systems that we consider is defined as follows.
Let $h_0$, $\mathcal{B}$, $\mathcal{T}_1$, $\mathcal{T}_2$, $\mathcal{C}_1$, $\mathcal{C}_2$, $\mathcal{K}$, $\nu_1$, $\nu_2$ and $\nu_3$ be positive real numbers, let $s_0\in(0,1]$, and let $g:\R\rightarrow\R$ be a function.
We define 
$\mathscr{C}(h_0,\mathcal{B},\mathcal{T}_1,\mathcal{T}_2,\mathcal{C}_1,\mathcal{C}_2,\mathcal{K},g,s_0,\nu_1,\nu_2,\nu_3)$
as the set of 5-tuples $(\mathscr{X}_h,\mathscr{A}_h,D(\mathscr{A}_h),\mathscr{B}_h,\mathscr{F}_h)$, where, for every $h\in [0,h_0)$:
\begin{itemize}
\item $\mathscr{X}_h$ is a Hilbert space (of finite or infinite dimension), endowed with the norm $\Vert\ \Vert_h$;
\item $\mathscr{A}_h:D(\mathscr{A}_h)\subset \mathscr{X}_h\rightarrow \mathscr{X}_h$ is a densely defined skew-adjoint operator;
\item $\mathscr{B}_h:\mathscr{X}_h\rightarrow \mathscr{X}_h$ is a bounded selfadjoint nonnegative operator such that $\Vert \mathscr{B}_h\Vert\leq \mathcal{B}$;
\item there exist $\mathscr{T}_h\in [\mathcal{T}_1,\mathcal{T}_2]$ and $\mathscr{C}_h\in [\mathcal{C}_1,\mathcal{C}_2]$ such that
$$
\mathscr{C}_h \Vert \phi_h(0)\Vert_h^2 \leq \int_0^{\mathscr{T}_h} \left(\Vert \mathscr{B}_h^{1/2}\phi_h(t)\Vert_h^2+h^\sigma\Vert \mathscr{V}_h^{1/2}\phi_h(t)\Vert_h^{2}\right)\, dt,
$$
for every solution of
$$
\phi_h'(t)+\mathscr{A}_h\phi_h(t)=0;
$$
\item $\mathscr{F}_h:\mathscr{X}_h\rightarrow \mathscr{X}_h$ is a mapping that is Lipschitz continuous on bounded subsets of $\mathscr{X}_h$, with Lipschitz constant less than $\mathcal{K}$, and satisfying \eqref{assumptionfrhof} and \eqref{ineqg} with the function $g$;
\item $g$ is an increasing odd function of class $C^1$ such that $g(0)=g'(0)=0$, $\lim_{s\rightarrow 0} sg'(s)^2/g(s) = 0$, and such that $s\mapsto \sqrt{s}g(\sqrt{s})$ is strictly convex on $[0,s_0^2]$;
\item $\mathscr{V}_h:\mathscr{X}_h\rightarrow \mathscr{X}_h$ is a positive selfadjoint operator, and the family
$(h^{\sigma/2} \Vert\mathscr{V}_h^{1/2}\Vert)_{h\in (0,h_0)} $ is uniformly bounded;
\item any solution $u_h$ of
\begin{equation*}
u_h'(t) + \mathscr{A}_hu_h(t) + \mathscr{B}_h \mathscr{F}_h(u_h(t)) + h^\sigma\mathscr{V}_h u_h(t) = 0 ,
\end{equation*}
with $u_h(0)\in D(\mathscr{A}_h)$, is well defined on $[0,+\infty)$, and satisfies
\begin{equation*}
\Vert u_h(t)\Vert_h^2 \leq \nu_3 \max(\nu_1,\Vert u_h(0)\Vert_h) \, L\left( \frac{1}{\psi^{-1}( \nu_2 t )} \right) ,
\end{equation*}
for every time $t\geq 0$, where $L$ and $\psi$ are defined (in function of $g$) by \eqref{defL} and \eqref{defpsir}.
\end{itemize}

Within the framework and notations introduced in Sections \ref{sec1} and \ref{sec_discretization}, under the assumptions of Theorems \ref{thm_continuous} and \ref{thm_space}, there exist positive real numbers $\mathcal{B}$, $\mathcal{T}_1$, $\mathcal{T}_2$, $\mathcal{C}_1$, $\mathcal{C}_2$, $\mathcal{K}$, $\nu_1$, $\nu_2$ and $\nu_3$, such that the 5-tuple $(X,A,D(A),B,F)$ and the one-parameter family of 5-tuples $(X_{\triangle x},A_{\triangle x},X_{\triangle x},B_{\triangle x},F_{\triangle x})$, $\triangle_x\in(0,\triangle x_0)$, belong to the class
$\mathscr{C}(\triangle x_0,\mathcal{B},\mathcal{T}_1,\mathcal{T}_2,\mathcal{C}_1,\mathcal{C}_2,\mathcal{K},g,s_0,\nu_1,\nu_2,\nu_3)$.

Here, the parameter $h$ used in the definition of the abstract class above stands for the space semi-discretization parameter $\triangle x$, whenever $h>0$, and if $h=0$ then we recover exactly the continuous setting of Section \ref{sec1}.

\medskip

We claim that Theorem \ref{thm_time} can be applied within this class, uniformly with respect to the parameter $h$. This can be checked straightforwardly, noticing in particular that the constants appearing in the estimates of that result depend only on the data defining the class.

\medskip

In particular, from that remark, we infer full discretization schemes of \eqref{main_eq}, in which we have first discretized in space, and then in time, obtaining an energy decay as in Theorems \ref{thm_continuous}, \ref{thm_space} and \ref{thm_time}, uniformly with respect to the discretization parameters $\triangle x$ and $\triangle t$. Said in other words, we apply Theorem \ref{thm_time} to the one-parameter family of systems given by \eqref{waveqSpaceDiscrete}, parametrized by $\triangle x$, and for which, by Theorem \ref{thm_space}, we already know that the solutions decay in a uniform way. We have thus obtained the following result.

\begin{theorem}\label{thm_full}
Under the assumptions of Theorems \ref{thm_continuous}, \ref{thm_space} and \ref{thm_time}, the solutions of
\begin{equation*}
\left\{\begin{split}
& \frac{\widetilde{u}^{k+1}_{\triangle x} - u^k_{\triangle x}}{\triangle t} + A_{\triangle x}\left(\frac{u^k_{\triangle x}+ \widetilde{u}^{k+1}_{\triangle x}}{2}\right) + B_{\triangle x}F_{\triangle x} \left( \frac{u^k_{\triangle x}+ \widetilde{u}^{k+1}_{\triangle x}}{2}\right) + \mathcal{V}_{\triangle x} \left(\frac{u^k_{\triangle x}+ \widetilde{u}^{k+1}_{\triangle x}}{2}\right) = 0,\\
& \frac{  \widetilde{u}^{k+1}_{\triangle x} - u^{k+1}_{\triangle x} }{\triangle t} = \mathcal{V}_{\triangle t} u^{k+1}_{\triangle x} , 
\end{split}\right.
\end{equation*}
are well defined for every integer $k$, for every initial condition $u^0_{\triangle x}\in X_{\triangle x}$, and for every $\triangle x\in (0,\triangle x_0)$, and the energy of any solution satisfies
\begin{equation*}
\frac{1}{2}\Vert u^k_{\triangle x}\Vert_{\triangle x}^2 = E_{u^{k}_{\triangle x}} \leq  T \max(\gamma_1, E_{u^0_{\triangle x}}) \, L\left( \frac{1}{\psi^{-1}( \gamma_2 k \triangle t )} \right),
\end{equation*}
for every integer $k$, 
with $\gamma_2\simeq  C_T / T ( 1 + e^{2 T\Vert B\Vert} \max( 1, T\Vert B\Vert ) ) $
and $\gamma_1 \simeq \Vert B\Vert / \gamma_2$.
Moreover, under \eqref{condlimsup}, we have the simplified decay rate
\begin{equation*}
E_{u^{k}_{\triangle x}} \leq   T \max(\gamma_1, E_{u^0_{\triangle x}}) \, (H')^{-1}\left( \frac{\gamma_3}{k \triangle t } \right),
\end{equation*}
for every integer $k$, for some positive constant $\gamma_3\simeq 1$.
\end{theorem}

\begin{example}
Let us consider the nonlinear damped wave equation studied in Section \ref{sec_nonlinearwave}. We make all assumptions mentioned in that section.

We first semi-discretize it in space by means of finite differences, as in Remark \ref{rk:commentsdiscobsineq}, with the viscosity operator $\mathcal{V}_{\triangle x} = - (\triangle x)^2\triangle_{\triangle x}$.
At some point $x_\sigma$ of the mesh ($\sigma$ being an index for the mesh), we denote by $u_{\triangle x,\sigma}(t)$ the point of $\R^n$ representing the approximation of $u(t,x_\sigma)$. Then the space semi-discrete scheme is given by
$$
u_{\triangle x,\sigma}''(t)-\triangle_{\triangle x} u_{\triangle x,\sigma}(t) + b_{\triangle x}(x_\sigma) \rho(x_\sigma,u_{\triangle x,\sigma}'(t)) - (\triangle x)^2\triangle_{\triangle x} u_{\triangle x}'(t) = 0 ,
$$
and the solutions of that system have the uniform energy decay rate $L(1/\psi^{-1}(t))$ (up to some constants).

Now, discretizing in time, we obtain the numerical scheme
\begin{equation*}
\left\{\begin{split}
& \frac{\widetilde{u}^{k+1}_{\triangle x,\sigma} - u^k_{\triangle x,\sigma}}{\triangle t} =  \frac{v^k_{\triangle x,\sigma}+\widetilde{v}^{k+1}_{\triangle x,\sigma} }{2}, \\
& \frac{\widetilde{v}^{k+1}_{\triangle x,\sigma} - v^k_{\triangle x,\sigma}}{\triangle t} - \triangle_{\triangle x,\sigma} \frac{u^k_{\triangle x,\sigma}+ \widetilde{u}^{k+1}_{\triangle x,\sigma}}{2} + b_{\triangle x,\sigma}(x_\sigma) \rho \left( x_\sigma, \frac{v^k_{\triangle x,\sigma}+\widetilde{v}^{k+1}_{\triangle x,\sigma} }{2} \right)  \\
& \qquad\qquad\qquad\qquad\qquad\qquad\qquad\qquad\qquad\qquad\qquad
 - (\triangle x)^2 \triangle_{\triangle x,\sigma} \frac{v^k_{\triangle x,\sigma}+\widetilde{v}^{k+1}_{\triangle x,\sigma} }{2} = 0,\\
& \frac{  \widetilde{u}^{k+1}_{\triangle x,\sigma} - u^{k+1}_{\triangle x,\sigma} }{\triangle t} = -(\triangle t)^2 \triangle_{\triangle x,\sigma}^2 u^{k+1}_{\triangle x,\sigma} , \qquad
\frac{  \widetilde{v}^{k+1}_{\triangle x,\sigma} - v^{k+1}_{\triangle x,\sigma} }{\triangle t} = -(\triangle t)^2 \triangle_{\triangle x,\sigma}^2 v^{k+1}_{\triangle x,\sigma} ,
\end{split}\right.
\end{equation*}
and according to Theorem \ref{thm_full}, the solutions of that system have the uniform energy decay rate $L(1/\psi^{-1}(t))$ (up to some constants).
\end{example}

\section{Conclusion and perspectives}
We have established sharp energy decay results for a large class of first-order nonlinear damped systems, and we have then studied semi-discretized versions of such systems, first separately in space and in time, and then as a consequence, for full discretizations. Our results state a uniform energy decay property for the solutions, the uniformity being with respect to the discretization parameters. This uniform property is obtained thanks to the introduction in the numerical schemes of appropriate viscosity terms.

Our results are very general and cover a wide range of possible applications, as overviewed in Section \ref{sec_examples}.

Now several questions are open, that we list and comment hereafterin.

\paragraph{(Un)Boundedness of $B$.}
The operator $B$ in \eqref{main_eq} has been assumed to be bounded, and this assumption has been used repeatedly in our proofs. This involves the case of local or nonlocal internal dampings, but this does not cover, for instance, the case of boundary dampings.
To give an example, let us consider the linear 1D wave equation with boundary damping
\begin{equation*}
\left\{\begin{array}{ll}
\partial_{tt}u-\partial_{xx} u = 0, & t\in(0,+\infty), \ x\in(0,1), \\
u(t,0)=0,\quad \partial_x u(t,1) = -\alpha \partial_t u(t,1), & t\in(0,+\infty),
\end{array}\right.
\end{equation*}
for some $\alpha>0$. It is well known that the energy $E(t)=\frac{1}{2}\int_0^1 \left( \vert\partial_t u(t,x)\vert^2 + \vert\partial_x u(t,x)\vert^2 \right) dx$ of any solution decays exponentially (see, e.g., \cite{BLR}).
It is proved in \cite{tebou_zuazua2} that the solutions of the regular finite-difference space semi-discrete model with viscosity
\begin{equation*}
\left\{\begin{split}
& u''_{\triangle x} - \triangle_{\triangle x} u_{\triangle x} - (\triangle x)^2 \triangle_{\triangle x} u'_{\triangle x} = 0, \\
& u_{\triangle x}(t,0)=0,\quad D_{\triangle x} u_{\triangle x}(t,1) = -\alpha \partial_t u(t,1), 
\end{split}\right.
\end{equation*}
where $D_{\triangle x}$ is the usual 1D forward finite-difference operator, given by
$$
D_{\triangle x} = \frac{1}{\triangle x}\begin{pmatrix}
-1 & 1 & \hdots & 0\\
0 & \ddots & \ddots  & \vdots\\
\vdots & \ddots & \ddots & 1\\
0 & \hdots & 0 & -1
\end{pmatrix} ,
$$
have a uniform exponential energy decay. Without the viscosity term, the decay is not uniform.
This example is not covered by our results, since the operator $B$ in that example is unbounded.

We mention also the earlier work \cite{BanksItoWang_1991}, in which the lack of uniformity had been numerically put in evidence for general space semi-discretizations of multi-D wave equations with boundary damping. By the way, the authors of that paper proposed a quite technical sufficient condition (based on energy considerations and not on viscosities) ensuring uniformity, and applied it to mixed finite elements in 1D and to polynomial Galerkin approximations in hypercubes. It is not clear if such considerations may be extended to our nonlinear setting.

\paragraph{More general nonlinear models.}
In relationship with the previous problem on $B$ unbounded, the question is open to treat equations like \eqref{main_eq}, but where the nonlinearity $F$ involves as well an unbounded operator, like for instance the equation
$$
u'(t)+Au(t)+BF(u(t),\nabla u(t)) = 0.
$$
There, the situation seems widely open. In our approach, what is particularly unclear is how to extend Lemma \ref{cont_lem3}.

An intermediate class of problems that has not been investigated at the discrete level is for instance the class of semilinear wave equations with strong damping
$$
\partial_{tt}u - \triangle u - a \triangle \partial_t u + b \partial_t u + g\star\triangle u + f(u) =0 ,
$$
with $f$ being not too much superlinear.
Here also, many variants are possible, with boundary damping (see \cite{alabau_AMO2005, alabauJDE}, with nonlocal terms (such as convolution), etc.
There exists a huge number of papers establishing decay rate results for such equations, see e.g. \cite{Dafermos, FGP} and the references therein
and see also \cite{AlabauCannarsa} for a nontrivial extension of the optimal-weight convexity method to the case of memory dissipation (non-local dissipation)
%
but nothing has been done for discretizations.

\paragraph{Geometric conditions and microlocal issues.}
Another class of equations of interest, not covered by our main result, is the stability of semilinear wave equations with locally distributed damping
$$
\partial_{tt}u - \triangle u + a(x) \partial_t u + f(u) =0 ,
$$
with Dirichlet boundary conditions. Here, $a$ is a nonnegative bounded function assumed to be positive on an open subset $\omega$ of $\Omega$, and the function $f$ is of class $C^1$, satisfying $f(0)=0$, $sf(s)\geq 0$ for every $s\in\R$ (defocusing case), $\vert f'(s)\vert\leq C\vert s\vert^{p-1}$ with $p\leq n/(n-2)$ (energy subcritical). We set $F(s)=\int_0^s f$.
It is proved in \cite{Zuazua_CPDE1990} (see also some extensions in \cite{DehmanLebeauZuazua,JolyLaurent_APDE2013} and a variant in \cite{AmmariDuyckaertsShirikyan}) that, under geometric conditions on $\omega$, the energy (which involves here an additional nonlinear term)
$$
\int_\Omega \left(  \frac{1}{2}(\partial_t u)^2 + \frac{1}{2}\Vert\nabla u\Vert^2 + F(u) \right) dx
$$
decays exponentially in time along any solution.
It is natural to expect that this exponential decay is kept in a uniform way for discrete models, if one adds appropriate viscosity terms as we have done in this paper. Note that this is not covered by our results since the above energy involves an additional nonlinear term.

Besides, in the general case where $\omega$ satisfies the Geometric Control Condition (GCC) of \cite{BLR}, the arguments of \cite{AmmariDuyckaertsShirikyan,DehmanLebeauZuazua,JolyLaurent_APDE2013} rely on microlocal issues and it is not clear whether or not such arguments may withstand discretizations. For instance, it is not clear what GCC becomes in a discrete setting. It is also not clear what a microlocal argument is at the discrete level, and such considerations may lead to several possible interpretations.
In brief, we raise here the completely open (and deliberately informal and imprecise) question:
\begin{quote}
\textit{Do microlocalization and discretization commute?}
\end{quote}

\paragraph{Uniform polynomial energy decay for linear equations without observability property}
Let us focus on linear equations, that is, let us assume that $F=\mathrm{id}_X$ in \eqref{main_eq}. 
It is well known that, in the continuous setting, the observability inequality \eqref{observability_assumption} holds true for all solutions of the conservative linear equation \eqref{cont_conservative}, if and only if the solutions of the linear damped equation \eqref{cont_lineardamped} (which coincides with \eqref{main_eq} in that case) have an exponential decay (see \cite{Haraux}).

If the observability inequality \eqref{observability_assumption} is not satisfied, then the decay of the energy cannot be exponential, however, it may be polynomial in some cases.
It may be so for instance for some weakly damped wave equations in the absence of geometric control condition (see \cite{BLR}) but also for indirect stabilization for coupled systems, that is when certain equations are not directly stabilized, even though the usual geometric conditions are satisfied (see \cite{Alabau1999, Alabau2002, ACK2002}).
In that case, it would be of interest to establish a uniform polynomial decay rate for space and/or time semi-discrete and full discrete approximations of \eqref{main_eq}.
In \cite{ANVE}, such results are stated for second-order linear equations (certain examples being taken from \cite{Alabau2002, ACK2002}) , with appropriate viscosity terms, and under adequate spectral gap conditions.

Extending this kind of result to a more general framework (weaker assumptions, full discretizations), and to our nonlinear setting, is open. 
Our strategy of proof is indeed strongly based on the use of observability inequalities.

\bigskip

\noindent{\bf Acknowledgment.}
The third author was partially supported by the Grant FA9550-14-1-0214 of the EOARD-AFOSR.

\end{document}